\documentclass{article}

\usepackage[english]{babel}
\usepackage[utf8]{inputenc} 
\usepackage[T1]{fontenc} 
\usepackage[a4paper, margin=0.8in]{geometry}

\usepackage{amssymb} 
\usepackage{amsthm} 
\usepackage{amsmath} 
\usepackage{mathtools} 
\usepackage{dsfont} 
\usepackage{mathrsfs} 
\usepackage{stmaryrd} 
\usepackage{eqnarray} 
\numberwithin{equation}{section} 

\usepackage{hyperref} 

\usepackage{enumitem} 
\usepackage{array} 
\usepackage{framed} 

\usepackage{tcolorbox} 
\usepackage{fancybox} 

\allowdisplaybreaks 

\newcommand{\noi}{\noindent} 

\newcommand{\prth}[1]{\left(#1 \right)} 

\newcommand{\gO}[2]{\mathcal{O}_{#2}\left(#1 \right)} 

\newcommand{\diff}{\mathrm{d}} 

\newcommand{\matp}[1]{\begin{pmatrix}#1\end{pmatrix}} 

\DeclareMathOperator{\Span}{Span}
\newcommand{\sgn}{\mathrm{sgn}} 
\DeclareMathOperator{\Le}{L^2}
\renewcommand{\L}{\Le}
\renewcommand{\H}{\mathrm{H}_\ep} 

\newcommand{\ep}{\varepsilon} 
\newcommand{\eps}{\ep} 
\newcommand{\E}{\mathrm E} 
\newcommand{\D}{\mathrm{D}} 
\newcommand{\m}{\mathrm{m}} 
\renewcommand{\i}{\mathrm{i}} 

\newcommand{\N}{\mathbb{N}} 
\newcommand{\Z}{\mathbb{Z}} 
\newcommand{\R}{\mathbb{R}} 
\newcommand{\C}{\mathbb{C}} 

\DeclareMathOperator{\WP}{WP}

\usepackage{pgfplots}
\pgfplotsset{compat=1.18}

\newtheorem{theorem}{Theorem}[section]
\newtheorem{lemma}[theorem]{Lemma}

\newtheorem{proposition}[theorem]{Proposition}

\theoremstyle{definition}
\newtheorem{remark}[theorem]{Remark}
\newtheorem*{rem*}{Remark}

\newtheorem*{def*}{Definition}

\usepackage[rightcaption]{sidecap}
\sidecaptionvpos{figure}{c}
\usepackage{graphicx} 

\title{Edge State Propagation Near Multiple and Singular Interfaces}

\author{NICOLAS FRANTZ AND \'ERIC VACELET}

\begin{document}

\maketitle

\begin{abstract}
We study the semiclassical propagation of edge states for a two-dimensional Dirac operator with a spatially varying smooth mass.
The zero set of the mass models the interface between topological phases and is allowed either to consist of two disjoint smooth curves or to possess an isolated singular point.
For disjoint interfaces, we prove that a wave packet initially localized on one interface remains confined near this interface over long semiclassical time scales and we quantify the influence of the distance between the two components on the accuracy of the approximation.
For singular interfaces, we determine the regime in which the wave-packet approximation remains valid as the packet approaches the singularity and identify the scaling at which our estimates lose their uniformity.
\end{abstract}

\noi \textit{Keywords} : propagation, Dirac operator, semiclassical analysis, wave packet expansion.

\noi $2024$ \textit{Mathematics Subject Classification} : 35B40, 35F05, 35Q40.

\tableofcontents

\section{Introduction}

\subsection{Time-dependent Dirac equation}

In this paper, we study the propagation of wave packets for the following two-dimensional semiclassical time-dependent Dirac equation: 
\begin{equation} \label{eq:Dirac}
	\begin{cases}
		\left( \ep \D_t + \m(x)\sigma_3 + \ep \D_1 \sigma_1 + \ep \D_2 \sigma_2\right) \psi^\ep (t,x) = 0, \quad (t,x)\in\R\times\R^2, \\
		\psi^\ep(0,\cdot) = \psi^\ep_0, 
	\end{cases}
\end{equation}
where the matrices $\displaystyle{ \prth{\sigma_j}_{1 \leqslant j \leqslant 3} }$ are the Pauli matrices:
\[
    \sigma_1 \coloneqq \begin{pmatrix} 0 & 1 \\ 1 & 0 \end{pmatrix}, \quad
    \sigma_2 \coloneqq \begin{pmatrix} 0 & -\i \\ \i & 0 \end{pmatrix}, \quad
    \sigma_3 \coloneqq \begin{pmatrix} 1 & 0 \\ 0 & -1 \end{pmatrix},
\]
and, for $\# \in \{t,1,2\}$, $\displaystyle{\D_\# \coloneqq -\i \partial_\#}$ with $\partial_j \coloneqq \partial_{x_j}$ for $j \in \{1,2\}$.
The semiclassical parameter $\ep$ is a small positive real number.
The family $\displaystyle{\prth{\psi^\ep_0}_{\ep > 0} }$ is uniformly bounded in $\L\prth{\R^2, \C^2}$.
The mass function $x \mapsto \m(x)$ is smooth and real-valued (so $\m \in \mathscr{C}^\infty(\R^2,\R)$), and all its derivatives are bounded.
The Hamiltonian for this system reads:
\[
	\H \coloneqq \begin{pmatrix}
		\m(x) & \eps \D_1-i\eps \D_2\\
		\eps \D_1+i\eps \D_2 & -\m(x)
	\end{pmatrix}.
\]

The semiclassical Dirac equation~\eqref{eq:Dirac} can be understood as an effective model arising from two-dimensional Schrödinger operators with periodic potentials possessing honeycomb symmetry~\cite{FeffermanWeinstein2012, FeffermanWeinstein2012JEDP}.
Such operators generically exhibit conical intersections in their band structure, known as Dirac points, with a two-dimensional eigenspace associated to each point.
When wave packets are spectrally localized near a Dirac point, a multiscale expansion shows that their dynamics at large scales is governed, at leading order, by a free Dirac equation for the slowly varying envelope of the Bloch modes~\cite{FeffermanWeinstein2012JEDP}.
Adding a perturbation that breaks inversion symmetry lifts the degeneracy at the Dirac point and opens a spectral gap; projecting this perturbation onto the degenerate eigenspace produces a diagonal term proportional to the matrix $\sigma_3$, giving rise to an effective mass $\m$.
If the perturbation varies slowly in space, then the mass $\m$ becomes position-dependent, and, introducing a semiclassical scaling $\ep$, leads precisely to the semiclassical Dirac operator~\eqref{eq:Dirac}, according to~\cite{Drouot2018}.

\bigskip

This framework naturally connects to the physics of topological insulators, which are materials characterized by an insulating bulk and conducting edge states protected by the topology of the electronic wavefunctions~\cite{HasanKane,ZhuCheng}.
In two-dimensional systems, such as quantum spin Hall insulators, the low-energy electronic properties are effectively described by a Dirac-type Hamiltonian where the sign of the mass term encodes the topological phase.
In our model, regions where $\m$ is positive or negative correspond to distinct topological phases, and the interface
\[
    \E \coloneqq \{x \in \R^2 \mid \m(x) = 0\},
\]
represents the boundary between these phases.
According to the bulk–edge correspondence, gapless Dirac fermions arise along this interface, giving rise to edge states robust against perturbations preserving the relevant symmetries.
The semiclassical parameter $\ep$ then allows one to study the propagation of wave packets along these edges in the adiabatic regime, capturing the dynamics of these topologically protected modes.

\bigskip

The propagation of wave packets for Dirac-type equations has been investigated in several complementary regimes.
In the free case arising from periodic media with Dirac points, wave packets spectrally localized near a fixed conical crossing propagate according to the effective Dirac dynamics derived in~\cite{FeffermanWeinstein2012JEDP,FeffermanWeinstein2012}.
These works provide a rigorous multiscale analysis showing that, over long but finite time scales, the envelope of Bloch modes evolves according to a two-dimensional Dirac equation.

\medskip

In the presence of a mass term, the opening of a spectral gap modifies the semiclassical transport in the bulk while preserving a Hamiltonian structure away from the crossing set.
The semiclassical analysis of Dirac operators with spatially varying mass has been developed in~\cite{Drouot2018}, where interface eigenvalues and their localization near the zero set of the mass are established using microlocal techniques.
This work provides a precise spectral description of domain-wall configurations in the semiclassical regime.

\medskip

The propagation of semiclassical wave packets for Dirac-type operators with variable mass has been the main topic of active investigation in recent years.
A significant contribution in this direction is the work of Bal et al.~\cite{Edge-States}, where families of wave packet solutions to Dirac Hamiltonians with slowly varying curved interfaces are explicitly constructed; these wave packets travel unidirectionally and remain dispersion-free along the curved edge, illustrating robust semiclassical edge dynamics in the presence of topological domain walls.
Complementing such constructive methods, Drouot~\cite{Drouot} provides a rigorous semiclassical analysis of remainder dispersive terms.
Then, the evolution of semiclassical measures has been achieved via a two-scale Wigner measure approach~\cite{Vacelet}.

\subsection{Geometric framework}

In this work, we extend the wave-packet construction of~\cite{Edge-States} to singular interfaces and analyze the stability of an edge state in presence of nearby other interfaces.
Our first result provides a uniform approximation for disjoint interfaces and identifies the critical dependence on the separation distance in order to preserve the edge state.
Our second result describes the breakdown of the wave-packet approximation near singular points where the transversality condition fails.

\medskip

We assume that the interface decomposes as $\E = \E_1\cup\E_2$, where $\E_1$ and $\E_2$ are smooth curves such that $\E_1\cap\E_2$ is of cardinal finite at most 1.

\medskip

To prove our main results we adapt the strategy of~\cite{Edge-States}.
Rather than being specific to the case of disjoint smooth interfaces, the same strategy naturally extends to our singular geometries.
More precisely, we consider, as singular geometries, that the zero set $\E$ loses regularity at a unique point denoted $x_\star$ such that $|\nabla \m(x_\star)| = 0$ while $|\nabla^2\m(x_\star)|\neq 0$.
To rule out pathological behaviors, such as the plateau function
$\displaystyle{\m(x) = x_1^2 - \exp\left( \frac{-1}{x_2^2} \right)}$, we assume there exists $K \in \N$ such that $\nabla^2 \m (x)$ is similar, near $x_\star$, to
\begin{equation} \label{eq:behavior_hess_m}
	\pm \begin{pmatrix} \kappa_\star & 0 \\ 0 & \ell \left( (x-x_\star)^K \right) \end{pmatrix},
\end{equation}
with $\ell$ a $K$-linear form and $\ell \left( (x-x_\star)^K \right)$ denotes the $K$-linear form $\ell$ evaluated on the family of $K$ instances of the same element: $x-x_\star$.
It notably implies that if $(x_t)_{t\in I}$ is a parametrization of $\E_1$ in a neighborhood $I$ of $t^\star$, with $x_{t^\star} = x_\star$, then
\begin{equation}\label{eq:order_blow_up}
    |\nabla \m (x_t)| = C|t^\star-t|^\frac{K+2}{2} + o \left( (t^\star-t)^{\frac{K+2}{2}} \right), \quad t\in I,
\end{equation}
with $C$ a positive constant.
This leads to the following different settings (see Figure~\ref{fig:settings1}) for the mass function $\m$, where the red arrows indicate the propagation direction along $\E \times\{0\}$ according to~\cite{Edge-States}.

\begin{figure}[!ht]
	\begin{center}
		\begin{tikzpicture}[scale=0.85]
			
			
			\draw (1,-2.4) node[left]{Setting A} ;
			
			\draw (-2,-2) -- (2,2);
			\draw (-2,2) -- (2,-2);
			\filldraw [blue] (0,0) circle (2pt) node[black, left]{$x_\star$};
			\filldraw [red] (-1,1) circle (2pt) node[left]{$x_0$};
			\draw[->, red, very thick] (-1,1) -- (-0.5,0.5);
			\draw[->, red, very thick] (1,1) -- (1.5,1.5);
			\draw[->, red, very thick] (-1,-1) -- (-1.5,-1.5);
			\draw[->, red, very thick] (1,-1) -- (0.5,-0.5);
			
			\draw (0.6,1.5) node[left]{$\m<0$} ;
			\draw (0.6,-1.5) node[left]{$\m<0$} ;
			\draw (1.5,0) node[right]{$\m>0$} ;
			\draw (-1.5,0) node[left]{$\m>0$} ;
			
			
			\draw (7,-2.4) node[left]{Setting B} ;
			
			\draw (7.5,-2) .. controls (5.5,-1) and (5.5,1) .. (7.5,2);
			
			\draw (4.5,-2) .. controls (6.5,0) .. (4.5,2);
			\filldraw [blue] (6,0) circle (2pt) node[black, left]{$x_\star$};
			\filldraw [red] (5,1.5) circle (2pt) node[left]{$x_0$};
			\draw[->, red, very thick] (5,1.5) -- (5.5,0.95);
			\draw[->, red, very thick] (6.8,1.54) -- (7,1.7);
			\draw[->, red, very thick] (5.5,-0.95) -- (5,-1.5);
			\draw[->, red, very thick] (7,-1.7) -- (6.8,-1.54);
			
			\draw (6.6,1.5) node[left]{$\m<0$} ;
			\draw (6.6,-1.5) node[left]{$\m<0$} ;
			\draw (6.5,0.5) node[right]{$\m>0$} ;
			\draw (5.5,0.5) node[left]{$\m>0$} ;
			
			
			\draw (13,-2.4) node[left]{Setting C} ;
			
			\draw (10,2) .. controls (11.9,1) .. (12,0);
			
			\draw (14,2) .. controls (12.1,1) .. (12,0);
			
			\filldraw [blue] (12,0) circle (2pt) node[black, left]{$x_\star$};
			\filldraw [red] (11,1.45) circle (2pt) node[below left]{$x_0$};
			\draw[->, red, very thick] (11,1.45) -- (11.5,1.17);
			\draw[->, red, very thick] (12.5,1.16) -- (13,1.45);
			
			\draw (12.6,1.5) node[left]{$\m<0$} ;
			\draw (12.6,-1.5) node[left]{$\m>0$} ;
			
		\end{tikzpicture}
		
		\caption{Different local behavior of geometries}
		\label{fig:settings1}
		
	\end{center}
\end{figure}

Setting A is the generic case, corresponding to an invertible Hessian  $\nabla^2 \m (x_\star)$, which is equivalent to $K = 0$ in~\eqref{eq:behavior_hess_m}.
In this configuration, there are two distinguished directions.
Since $\m$ changes sign across the interface $\E$, the Hessian must have signature $(1,-1)$.
Equivalently, $\nabla^2 \m (x_\star)$ is similar, to
\[
    \pm \begin{pmatrix} \kappa_\star & 0 \\ 0 & \kappa_2 \end{pmatrix},
\]
where $\kappa_\star\in\R\backslash\lbrace 0\rbrace$, and $\kappa_2$ is a non-zero real number with opposite sign to $\kappa_\star$.
A typical example of Setting A is obtained with the mass function $\m_\mathrm{A}(x) = x_1^2 - x_2^2$.

\medskip

Settings B and C correspond to the case where $\nabla^2 \m (x_\star)$ is not invertible.
Setting B arises when $K$ is even in~\eqref{eq:behavior_hess_m}.
Compared with Setting A, the intersection is preserved while one of the quadratic directions is flattened.
A typical example of setting B is given by $\m_\mathrm{B}(x) = x_1^2 - x_2^4$.
Setting C corresponds to an odd value of $K$ in~\eqref{eq:behavior_hess_m}.
In contrast with Setting B, the singularity creates a turning point for the propagation of the wave packet.
An example of this situation is achieved with $\m_\mathrm{C}(x) = x_1^2 - x_2^3$.
These configurations represent singular perturbations of regular interfaces and may be viewed as local normal forms for degenerate domain walls.

\bigskip

Setting B has an additional feature.
It can be viewed as a degenerate limit of the family of Configuration B$_\delta$ introduced above, obtained by opening the singular point with a separation parameter $\delta>0$.
In this way, one recovers two disjoint interfaces, while Setting B corresponds to the limit $\delta\to 0$ (see Figure~\ref{fig:settings}).

\begin{figure}[!ht]
	\begin{center}
		\begin{tikzpicture}
			
			
			\draw (1.1,-2.4) node[left]{Setting B$_\delta$} ;
			
			\draw (-2,-2) .. controls (0,0) .. (-2,2);
			\filldraw [blue] (-0.5,0) circle (2pt);
			
			\draw (2,-2) .. controls (0,-1) and (0,1) .. (2,2);
			\filldraw [blue] (0.5,0) circle (2pt);
			
			\draw (0.6,-1) node[left]{$\m<0$} ;
			\draw (1.5,0.5) node[right]{$\m>0$} ;
			\draw (-1.5,0.5) node[left]{$\m>0$} ;
			
			\draw[<->] (-0.4,0) -- (0.4,0) ;
			\draw (0.2,0.3) node[left]{$\delta$} ;
			
			
			\draw (7.1,-2.4) node[left]{Setting B} ;
			
			\draw (4.5,-2) .. controls (6.5,0) .. (4.5,2);
			
			\draw (7.5,-2) .. controls (5.5,-1) and (5.5,1) .. (7.5,2);
			\filldraw [blue] (6,0) circle (2pt);
			
			\draw (6.6,1.5) node[left]{$\m<0$} ;
			\draw (6.6,-1.5) node[left]{$\m<0$} ;
			\draw (6.5,0.5) node[right]{$\m>0$} ;
			\draw (5.5,0.5) node[left]{$\m>0$} ;
			
		\end{tikzpicture}
		
		\caption{"Blowing up" of Setting B}
		\label{fig:settings}
		
	\end{center}
\end{figure}

We first show how, for disjoint curves $\E_1$ and $\E_2$ (Setting B$_\delta$), a wave packet initially prepared as an edge state along $\E_1$ remains localized near $\E_1$ over long semiclassical time scales.
The error estimate depends explicitly on the separation distance $\delta$ between the two interfaces and shows that the interaction becomes negligible when the curves are sufficiently far apart, according to the semiclassical parameter.

\medskip

We then show that, for crossed curves (Settings A, B, C), a wave packet initially prepared as an edge state along $\E_1 \setminus (\E_1\cap\E_2)$ remains localized near $\E_1$ over long semiclassical time scales.
The error estimate depends explicitly on the behavior of the Hessian of $\m$, when the curvature does not vanished at the singularity.

\subsection{Main results}

We now recall few basic facts on wave packets and their semiclassical evolution.
For any $\vec{f}\in\L(\R^2,\C^2)$ and any $(x_0,\xi_0)\in\R^2_{x}\times\R^2_{\xi}$, the family of wave packets with profile $\vec{f}$,
\[
    \left(\WP^\varepsilon_{(x_0,\xi_0)}\left[\vec{f}\right]\right)_{\varepsilon>0}:\R^2\to\R^2,
\]
is defined for all $\varepsilon>0$ by, for all $x \in \R^2$, 
\[
    \WP^\varepsilon_{(x_0,\xi_0)}\left[\vec{f}\right](x)=\frac{\exp\left({i\varepsilon^{-1}\xi_0\cdot(x-x_0)}\right)}{\sqrt{\varepsilon}}\vec{f}\left(\frac{x-x_0}{\sqrt{\varepsilon}}\right),
\]
where $(x\cdot y)$ denotes the usual Euclidean inner product between $x$ and $y$.
These wave packets are concentrated in the phase space $\mathrm{T}^*\R^2$ near $(x_0,\xi_0)$ at scale $\sqrt{\varepsilon}$.
Since edge states are associated with the characteristic manifold where the matrix valued symbol of $\H$ ceases to be invertible, all the wave packets considered in this work are centered at points $(x_t,0)$ belonging to this characteristic set.

\bigskip

We now state our main theorem, for Configuration B$_\delta$.

\begin{theorem}\label{thm:distance}
Suppose that $\E_1\cap\E_2 = \emptyset$.
Let $1>\delta>0$ denote the distance between $\E_1$ and $\E_2$ and let $(x_t)_{t\in\R}$ be an arc-length parametrization of $\E_1$.
Then, there exists a family of functions $(a_k)_{k\in\N}$ that belong to $\mathscr{C}^\infty(\R,\mathscr{S}(\R^2,\C^2))$ such that for all $n\in\N$, there exists a constant $C_n>0$ such that for all $t\in\R$,
	\[
		\left\|\Psi(t,\cdot)-\WP_{(x_t,0)}^\varepsilon\left[\sum_{k=0}^n\varepsilon^{k/2}a_k(t,\cdot)\right]\right\|_{\L(\R^2,\C^2)}\leq C_n |t| \left(\frac{\sqrt{\varepsilon}}{\delta^3}\right)^n\left(1+\left(\frac{\varepsilon}{\delta}\right)^{\frac{n}{2}}\right),
	\]
where $\Psi$ is the solution of~\eqref{eq:Dirac} with initial data $\Psi_0 = \displaystyle\WP_{(x_0,0)}\left[\sum_{k=0}^n \varepsilon^{k/2}a_k(0,\cdot)\right]$.
\end{theorem}

In particular, for all $\alpha>0$, if $\delta\gg\varepsilon^{1/6+\alpha}$ our theorem establishes that the edge state is stable and not affected by the presence of the other interface $\E_2$, while $t$ is small.

\medskip

We then state our main theorem for Settings A, B and C.

\begin{theorem}\label{theo:main}
    Suppose that $\E_1\cap\E_2 = \{x_\star\}$ such that $\nabla^2 \m (x)$ satisfy~\eqref{eq:behavior_hess_m} near $x_\star$.
	Let $(x_t)_{t\in[0,t^\star)}$ be a local arc-length parametrization of $\E$ satisfying~\eqref{eq:order_blow_up}.
	Then, there exists a family of functions $(a_k)_{k\in\N}$ that belong to $\mathscr{C}^\infty(\R,\mathscr{S}(\R^2,\C^2))$ such that for all $n\in\N$, there exists a constant $C_n>0$ such that for all $t\in [0,t^\star)$, 
	\[
		\left\|\Psi(t,\cdot)-\WP_{(x_t,0)}\left[\sum_{k=0}^n\varepsilon^{k/2}a_k(t,\cdot)\right]\right\|_{\L} \leq
		C_n \left( \frac{\sqrt{\varepsilon}}{(t^\star-t)^{2+\frac{3}{2}K}}\right)^n\left(1+\left(\frac{\varepsilon}{(t^\star-t)^{\frac{K+2}{2}}}\right)^\frac{n}{2}\right),
	\]
where $\Psi$ is the solution of~\eqref{eq:Dirac} with initial data $\Psi_0=\displaystyle\WP_{(x_0,0)}\left[\sum_{k=0}^n \varepsilon^{k/2}a_k(0,\cdot)\right]$.
\end{theorem}

\begin{remark}
	The result stays valid if we consider a smooth curve $\E$ such that $\nabla\m(x_\star)=0$ but remains a manifold on $x_\star$.
	In that case the result of Theorem~\ref{theo:main} applies.
\end{remark}

This result characterizes the scale at which a semiclassical wave packet ceases to accurately describe the evolution as it approaches a singular point of the interface.
The approximation remains valid as long as the trajectory remains at a distance from the singularity larger than $\sqrt{\varepsilon}$, measured through the degeneracy rate of $|\nabla \m(x_t)|$.
When this distance becomes comparable to this scale, our estimates no longer provide a uniform control of the wave packet approximation.
This loss of uniform control suggests that a different effective description may be required in a neighborhood of the singularity.
Our approximation could identify the correct scaling regime relevant to the dynamics near the singularity.

\subsection{Edge states structures}\label{subsub:BBDFLW}

If $\nabla \m$ does not vanish, then each connected component of $\E$ is a regular one-dimensional manifold.
In that case, for any $x_0\in\E$, the solution $(x_t)_{t\in\R}$ to the Ordinary Differential Equation:
\begin{equation} \label{eq:ODE}
\begin{cases}
    \dot x_t = \dfrac{\nabla \m(x_t)^\perp}{\left| \nabla \m(x_t) \right|}, \\
    x(t=0)= x_0,
\end{cases}
\end{equation}
defines a smooth arc-length parametrization of the corresponding branch of the interface.
Using polar coordinates, for each $t$, there exists $(r_t,\theta_t)\in \R_+^* \times \R$ such that
\begin{equation}\label{eq:deftheta}
    \nabla \m(x_t) = r_t 
    \begin{pmatrix}
        -\sin \theta_t \\
        \cos \theta_t
    \end{pmatrix},\quad \text{so} \quad 
    \dot{x}_t = - 
    \begin{pmatrix}
        \cos \theta_t \\
        \sin \theta_t
    \end{pmatrix},\quad \text{ and } r_t=|\nabla \m(x_t)|.
\end{equation}
The quantity $r_t = |\nabla \m(x_t)|$ quantifies the transverse strength of the confinement, while $\theta_t$ determines the local orientation of the interface.
For any $x_0\in\E$, it has been proved in~\cite[Theorem 2]{Edge-States} that if the initial data of~\eqref{eq:Dirac} is a wave packet of the form
\begin{equation} \label{eq:Initial_WP}
    \psi_0^\ep(x)
    =
    \frac{r_0^{1/4}}{\sqrt\ep}
    f_{\mathrm{in}} \left( \frac{1}{\sqrt\ep} \big(R_{\theta_0} (x-x_0)\big)_1 \right)
    \exp\!\left(-\frac{r_0}{2\varepsilon} \big(R_{\theta_0}(x-x_0)\big)_2^2\right)
    \vec V_{\theta_0},
\end{equation}
where $(x)_{j}$ denote the $j^\mathrm{th}$ coordinate of $x$, with $\displaystyle{R_\theta \coloneqq \begin{pmatrix} \cos(\theta) & \sin(\theta) \\ -\sin(\theta)& \cos(\theta) \end{pmatrix} }$, $f_{\mathrm{in}} \in \mathscr S(\R)$, and $\displaystyle{\vec V_{\theta_t} \coloneqq \begin{pmatrix} e^{-i\theta/2} \\ -e^{i\theta/2} \end{pmatrix}}$, then the solution of~\eqref{eq:Dirac} remains a semiclassical wave packet localized near $(x_t,0) \in \mathrm{T}^\ast \R^2$ at time $t$, namely
\begin{equation} \label{eq:solBBDF}
    \psi_t^\ep(x) =
    \frac{r_t^{1/4}}{\sqrt\ep}
    f_{\mathrm{in}} \left( \frac{1}{\sqrt\ep} \big(R_{\theta_t}(x-x_t)\big)_1 \right)
    \exp\!\left(-\frac{r_t}{2\varepsilon} \big(R_{\theta_t}(x-x_t)\big)_2^2\right)
    \vec V_{\theta_t}
    + \mathcal{O}_{\L(\R^2,\C^2)} \left(\sqrt{\ep}\langle t \rangle \right).
\end{equation}

In particular, this description is local in nature and relies crucially on the non-degeneracy of the interface: it breaks down in regions where $|\nabla \m|$ becomes small or vanishes, in particular near singular points of the zero set.

\bigskip

Wave packets have played a fundamental role in semiclassical analysis since the pioneering works of Hagedorn, who introduced a class of Gaussian coherent states with polynomial corrections, now known as Hagedorn wave packets~\cite{Hagedorn1980,Hagedorn1998,RobertCombescure}.
These functions provide accurate asymptotic solutions to the Schrödinger equation and yield a precise description of quantum dynamics in terms of underlying classical trajectories.
More precisely, in the scalar setting, the propagation of a coherent state follows the Hamiltonian flow associated with the principal symbol of the Hamiltonian.
Their stability properties and rich algebraic structure make Hagedorn wave packets a powerful tool for both theoretical analysis and numerical applications~\cite{LasserLubich}.

\medskip

The wave packet approach has subsequently been extended to matrix-valued Hamiltonians exhibiting eigenvalue crossings.
In particular, Hagedorn studied the propagation of wave packets through eigenvalue crossings~\cite{Hag94}, while Hagedorn and Joye described non-adiabatic transitions across small spectral gaps in the Landau--Zener regime~\cite{HJ1}.
More generally, the works~\cite{HagedornJoye,FermanianLasser,Curely,FermanianGambleHari} consider matrix-valued Hamiltonians under generic assumptions, notably requiring a non-vanishing trace of the Hamiltonian together with suitable transversality conditions.
In this framework, one constructs approximate wave packet solutions propagating along branches of Hamiltonian trajectories associated with the eigenvalues of the principal symbol.
These results provide a refined description of quantum dynamics near eigenvalue crossings and conical intersections.

\medskip

The present work addresses a more degenerate situation, for which the previous geometric picture no longer applies.
In our setting, the degeneracy of the Hamiltonian prevents the trajectories followed by the approximate wave packet solutions from being directly interpreted as classical Hamiltonian trajectories associated with the eigenvalues.
The construction of the relevant propagation curves therefore requires a different analysis.

\bigskip

To lighten notations, in the following, we set
\begin{equation}\label{eq:def Mstar}
	M_\star \coloneqq M_\star(K)=\frac{K+2}{2},
\end{equation}
which notably implies that~\eqref{eq:order_blow_up} rewrites
\[
    |\nabla \m (x_t)| = C|t^\star-t|^{M_\star} + o \left( (t^\star-t)^{M_\star} \right).
\]

\section{Strategy} \label{subsec:sketch}

In this section, we explain step by step how  Theorems~\ref{thm:distance} and~\ref{theo:main} are obtained.
First, we solve the hierarchy of equations obtained by a normal form procedure and a wave packet expansion.
Then, to compute $\L$-norm estimate of the approximation, we need to compute smooth time-dependent Schwartz estimates of the wave packet expansion, due to the hierarchy of equations.
Finally we derive our main theorem by using an energy estimate, which establishes how close our wave packet approximation is to the exact solution.
The guiding principle of the strategy is to refine the construction and estimates of~\cite{Edge-States}.

\bigskip

From now on, we fix a constant $T>0$ such that $r_t\neq 0$ for all $t\in [0,T)$.
In the setting of Theorem~\ref{thm:distance} (Setting B$_\delta$), $T$ can be chosen arbitrary large.
In the setting of Theorem~\ref{theo:main} (which correspond to the configuration A, B or C), we set $T=t^\star$.

\subsection{Hierarchy of equations} 

\subsubsection{Transport equations}

Looking for a solution to~\eqref{eq:Dirac} in the form of a semiclassical wave packet microlocalized near $(x_t,0)$ in phase space, we follow~\cite{Edge-States} by deriving formally the family of equations from the wave packet expansion.
We consider a polynomial expansion in $\sqrt{\ep}$ of the profile of the wave packet.
More precisely, we consider $\WP^\ep_{x_t,0} \left[ a^{(n)} \right]$, where for $n \in \N$,
\begin{equation} \label{eq:a^N}
	a^{(n)} \coloneqq \sum_{j=0}^n \eps^{j/2} a_j,
\end{equation}
with $(a_j)_{j \in \llbracket 0,n\rrbracket}$ in $ \mathscr{C}(\R,\mathscr{S}(\R^2,\C^2))$.
Then, we introduce local coordinates around the trajectory $x_t$ and we expand the mass function $\m$ in Taylor series at $x_t$.
After inserting the corresponding ansatz into~\eqref{eq:Dirac}, the equation can be rewritten as a function-valued polynomial in $\ep$ as follows,
\begin{equation} \label{eq:WP4}
	\sum_{k=0}^n \ep^{\frac{k}{2}} \sum_{j=0}^k T_j^t a_{k-j} + \ep^{\frac{n+1}{2}} \left( \sum_{k = 1}^{n} \sum_{j = n+1-k}^n \ep^{\frac{k+j}{2}} T_k^t a_j + \sum_{j=0}^n \ep^{\frac{j}{2}} R_t^{n+1} \sigma_3 a_j \right) = 0,
\end{equation}
where
\begin{align}
	T_0^t &\coloneqq -\dot{x}_t \cdot \D_x + 
	\begin{pmatrix} \nabla \m(x_t) \cdot x & \D_1 - \i\D_2 \\ \D_1 + \i\D_2 & -\nabla \m(x_t) \cdot x \end{pmatrix},\label{eq:T_O} \\
	T_1^t &\coloneqq \D_t + \frac12 
	\begin{pmatrix} \nabla^2 \m(x_t) \, x^2 & 0 \\ 0 & - \nabla^2 \m(x_t) \, x^2 \end{pmatrix}\label{eq:T_1}, \\
	T_k^t &\coloneqq \frac{1}{(k+1)!} 
	\begin{pmatrix} \nabla^{k+1} \m(x_t) \, x^{k+1} & 0 \\ 0 & - \nabla^{k+1} \m(x_t) \, x^{k+1} \end{pmatrix}, \quad k \geqslant 2 \label{eq:T_k}, \\
	R_t^{n+1}(y) & \coloneqq \int_0^1\frac{(1-u)^{n+1}}{(n+1) !}\nabla^{n+2}\m(x_t+u\sqrt{\eps}y).y^{n+2} \diff u. \label{eq:R_t}
\end{align}
Identifying the $n+1$-first terms at each order in $\sqrt{\ep}$ yields the hierarchy of equations
\begin{equation} \label{eq:transport}
	\sum_{j=0}^k T_j^t a_{k-j} = 0, \quad \text{for} \ k \in \llbracket 0,n \rrbracket,
\end{equation}
which governs the evolution of the profile components $a_k$.
Here, $T_0^t$ captures the linearized Dirac dynamics along the interface, $T_1^t$ accounts for time derivative and curvature corrections (through quadratic terms), and $T_k^t$ represents higher-order contributions of the Taylor expansion of $\m$ around $x_t$.
Using polar coordinates as in~\eqref{eq:deftheta}, we may observe that for any $f\in\mathscr{S}(\R,\C)$, the initial data $\psi_0^\varepsilon$ defined in~\eqref{eq:Initial_WP} belongs to the kernel of $T_0^t$.
We will even prove (consequence of Proposition~\ref{theo:spectrumt}) that 
\[
    \ker(T_0^t) = \left\lbrace \psi_0^\ep(x) = \frac{r_0^{1/4}}{\sqrt\ep} f \left( \frac{(\mathrm R_{\theta_0}(x-x_0))_1}{\sqrt\ep} \right) \exp\!\left(-\frac{r_0 (\mathrm R_{\theta_0}(x-x_0))_2^2}{2\ep}\right) \vec V_{\theta_0} \mid f\in\mathscr{S}(\R)\right\rbrace.
\]

\subsubsection{Hierarchy of normal form equations}

The resolution of~\eqref{eq:transport} is the main topic of Section~\ref{sec:solve-transport}.
An important feature is the analysis of the operators $T_k^t$ (especially $T_0^t$) which is derived after a reduction process.
To proceed, we introduce the following unitary operator acting on $\L(\R^2,\C^2)$, defined, for all $u \in \L(\R^2,\C^2)$, for all $y \in \R^2$ and for all $t\in [0,T)$ by
\begin{equation}\label{eq:def U_t}
    \mathscr{U}_t u(y) \coloneqq \frac{-i}{\sqrt2} \sqrt{r_t} \begin{pmatrix} e^{-i\frac{\theta_t}{2}} & e^{- i\frac{\theta_t}{2}} \\ e^{ i\frac{\theta_t}{2}} & -e^{i\frac{\theta_t}{2}} \end{pmatrix} u \left( \sqrt{r_t} \begin{pmatrix} \cos(\theta_t) & \sin(\theta_t) \\ -\sin(\theta_t) & \cos(\theta_t) \end{pmatrix} y \right).
\end{equation}
In particular, it hollows us to rewrite solution~\eqref{eq:solBBDF} for all $\ep \in (0,1]$ and for all $(t,x) \in \R \times \R^2$ as follows
\[
    \psi_t^\ep(x) = \WP_{x_t}^\varepsilon \left[\mathscr{U}_t \left( \mathbf{p}_t^\dagger f_{\mathrm{in}}(x_1) \otimes \exp\left(\frac{x_2^2}{2}\right) \begin{pmatrix} 0 \\ 1 \end{pmatrix} \right) \right] \ + \ \gO{\ep^{1/2} \langle t \rangle}{\L\left( \R^2,\C^2 \right)},
\]
where $\otimes$ denotes the tensor product between $\L(\R,\C)$ and $\L\left( \R,\C^2 \right)$ and $\mathbf{p}_t^\dagger$ is the adjoint of the unitary operator acting on $\L(\R,\C)$ defined by, for all $f \in \L(\R,\C)$ and for all $x \in \R$,
\[ 
    (\mathbf{p}_t f)(x) \coloneqq r_t^{1/4} f \left( \sqrt{r_t} x \right).
\]
Moreover, the initial condition~\eqref{eq:Initial_WP} becomes, for all $\varepsilon > 0$ and for all $x \in \R^2$,
\[
    \psi^\ep_0(x) = \WP^\ep_{x_0} \Big[ \mathscr{U}_0  \Big( \mathbf{p}_0^\dagger f_{\mathrm{in}} \otimes g_0 \Big) \Big] (x),
\]
where $f_{\mathrm{in}} \in \mathscr{S}(\R)$ and $g_0$ is defined as follows
\[
    \begin{array}{ccccc}
        g_0: & \R & \longrightarrow & \R^2 \\
         & x & \longmapsto & \displaystyle{\frac{1}{\sqrt{2\pi}} e^{-x^2/2}} \begin{pmatrix} 0 \\ 1 \end{pmatrix}.
    \end{array}
\]
Conjugating the operator $T_0^t$ defined in~\eqref{eq:T_O} by $\mathscr{U}_t$, we simplify its expression, according to the following proposition.

\begin{proposition} \label{prop:TA}
We have
\begin{align*}
    \mathscr{U}_t^{-1} T_0^t \mathscr{U}_t & = \sqrt{r_t} A_0,
\end{align*}
where the operator $A_0$ acts on $\L(\R^2,\C^2)$ and is defined by
\begin{equation} \label{eq:computeA_0}
    A_0 \coloneqq \begin{pmatrix} 2 \D_1 & x_2 + \i\D_2 \\ x_2 - \i\D_2 & 0 \end{pmatrix}.
\end{equation}
\end{proposition}

The operator $A_0$ is self-adjoint with a compact resolvent therefore we will consider, with an abuse of notation, the operator $A_0^{-1}$ defined as the inverse of $A_0$ restricted on the range of $A_0$.
This decomposition allows to encode the dependency in time in the unitary transformation and in the factor $\sqrt{r_t}$.
Moreover for all $k \in \N\setminus\{0\}$, for all $t\in [0,T)$, we denote by
\begin{equation} \label{eq:Atm}
    A_k^t \coloneqq \frac{1}{\sqrt{r_t}} \mathscr{U}_t^{-1} T_k^t \mathscr{U}_t.
\end{equation}
In the same way, we transform $R_t^{n+1}$ by conjugation in $\widetilde{R}_t^{n+1} \coloneqq r_t^{-1/2} \mathscr{U}_t^{-1} R_t^{n+1} \mathscr{U}_t$.
In the following proposition, we explicit the expression of $(A_k^t)_{k \geqslant 1}$ and $\widetilde{R}_t^{n+1}$.

\begin{proposition} \label{prop:TAA}
For all $t\in [0,T)$, we have
\begin{align}
    A_1^t & = \frac{1}{\i \sqrt{r_t}} \left( \frac{\dot{r}_t}{2r_t} (1 + x \cdot \nabla) + \partial_t - \dot{\theta}_t x^\perp \cdot \nabla + \frac{i\dot{\theta}_t }{2} \sigma_1 \right) + \frac{1}{2 r_t^{3/2}} \nabla^{2} \m(x_t) \left(R_{\theta_t}^\dagger x\right)^{2} \sigma_1, \label{eq:computeA_1} \\
    A_k^t & = \frac{1}{r_t^{\frac{k}{2}+1}} \frac{1}{(k+1)!} \nabla^{k+1} \m(x_t) \left(R_{\theta_t}^\dagger x\right)^{k+1} \sigma_1, \quad k \geqslant 2. \label{eq:computeA_k}
\end{align}
where $x^\perp =-i\sigma_2 x$.
Moreover, for all $t\in [0,T)$, for all $n \in \N$,
\begin{equation}
    \mathscr{U}_t^{-1} R_t^{n+1} \mathscr{U}_t = \frac{1}{r_t^{\frac{n+2}{2}}} \int_0^1 \frac{(1-u)^{n+1}}{(n+1) !} \nabla^{n+2} \m \left( x_t+u\sqrt{\eps} R_{\theta_t}^\dagger x \right).\left(R_{\theta_t}^\dagger x\right)^{n+2} \diff u. \label{eq:computeR_n}
\end{equation}
\end{proposition}

\subsubsection{Wave packet expansion normal form}
According to the unitary transformation $\mathscr{U}_t$, the family $(a_k)_{k \in \llbracket 0,n \rrbracket}$ is solution of~\eqref{eq:transport} if and only if, for all $k \in \llbracket 0,n \rrbracket$, 
\begin{equation} \label{eq:A_jj1}
    \sum_{j=0}^k A_j^t \mathscr{U}_t^{-1} a_{k-j} = 0,
\end{equation}
with the initial condition, for all $x \in \R^2$,
\[
\begin{cases}
    a_0(0,x) = \mathscr{U}_0  \Big( \mathbf{p}_0^\dagger f_{\mathrm{in}} \otimes g_0 \Big) (x), \\
    a_j(0,\cdot) = 0, & \text{for all} \ k \in \llbracket 1, k\rrbracket.
\end{cases}
\]
The spectral analysis of $A_0$ is a key point in the decomposition of wave packet expansion and in estimates.
Especially, we have (see Proposition~\ref{theo:spectrumt})
\begin{equation} \label{eq:kerA}
    \ker (A_0) = \L \left( \R_{x_1},\C \right) \otimes \Span{(g_0)}, \quad \text{and} \quad \ker (A_0)^\perp = \L \left( \R_{x_1},\C \right) \otimes \Span{(g_0)}^\perp.
\end{equation}
In order to show the existence of the family $(a_n)_{n \in \N}$ solution of~\eqref{eq:A_jj1}, we will consider a decomposition of $\left(\mathscr{U}_t^{-1}a_n\right)_{n \in \N}$ on $\ker(A_0) \bigoplus \ker(A_0)^\perp$.
More precisely, for all $n \in \N$, we write
\begin{equation} \label{eq:defauf}
    a_n = \mathscr{U}_t \Big( u_n + \mathbf{p}_t^\dagger f_n \otimes g_0 \Big).
\end{equation}
where $f_n \in \mathscr{C}(\R,\L(\R,\C))$ and $u_n \in \mathscr{C}(\R, \ker(A_0)^\perp)$.
According to the expression of $\ker(A_0)$ in~\eqref{eq:kerA} and decomposition~\eqref{eq:defauf}, Equation~\eqref{eq:A_jj1} is reduced to
\begin{equation} \label{eq:A_j}
    A_0 u_n + \sum_{k=1}^n A_k^t \Big( u_{n-k} + \mathbf{p}_t^\dagger f_{n-k} \otimes g_0 \Big) = 0.
\end{equation}
with the initial condition
\begin{equation} \label{eq:bad_initial_condition}
\begin{cases}
    \Big( u_0 + \mathbf{p}_0^\dagger f_0 \otimes g_0 \Big) (0,\cdot) = \mathbf{p}_0^\dagger f_{\mathrm{in}} \otimes g_0 \\
    \Big( u_n + \mathbf{p}_0^\dagger f_n \otimes g_0 \Big) (0,\cdot) = 0, & \text{for all} \ n \in \N\setminus\{0\},
\end{cases}
\end{equation}
which is a change of unknown functions.
Notably, if we want a solution at order $N \in \N$, then it suffices to determine $(f_n)_{n \in \llbracket 0,N-1 \rrbracket}$ and $(u_n)_{n \in \llbracket 0,N \rrbracket}$ since $f_N$ does not play a role in~\eqref{eq:A_j} when $n \in \llbracket 0,N \rrbracket$.
Since this family of equations imposed some very specific assumptions on $f_{\mathrm{in}}$ (the value at $t=0$ of the family $(u_n)_{n \in \N\setminus\{0\}} \in \mathscr{C}(\R, \ker(A_0)^\perp)^\N$ is determined by the function $f_{\mathrm{in}}$ and so does not necessarily vanish at $t=0$), then we will consider the family of equations defined as
\begin{equation}\label{eq:goodsystem}
    \begin{cases}\displaystyle{A_0 u_n + \sum_{k=1}^n A_k^t \Big( u_{n-k} + \mathbf{p}_t^\dagger f_{n-k} \otimes g_0 \Big) = 0}\\
    \Big( u_0 + \mathbf{p}_0^\dagger f_0 \otimes g_0 \Big) (0,\cdot) = \mathbf{p}_0^\dagger f_{\mathrm{in}} \otimes g_0 \\
    f_n(0,\cdot) = 0, & \text{for all} \ n \in \N\setminus\{0\}.
\end{cases}
\end{equation}
We construct these families by induction, which is the main result of Section~\ref{sec:solve-transport}.

\begin{proposition} \label{prop:decomposition}
Let $u_0 = 0$ and $f_0(t,\cdot) = f_{\mathrm{in}}$.
Then the families $(f_n)_{n \in \N}$ and $(u_n)_{n \in \N}$ satisfy~\eqref{eq:goodsystem}, if and only if, for all $n \in \N\setminus\{0\}$ and for all $t \in [0,T)$,
\begin{align}
    u_n & = - A_0^{-1} \left( \sum_{k=1}^n A_k^t \Big( u_{n-k} + \mathbf{p}_t^\dagger f_{n-k} \otimes g_0 \Big) \right), \label{eq:defu} \\
    f_n(t,\cdot) & = \frac{1}{i} \int_0^t \sqrt{r_s}\mathbf{p}_s \left\langle \sum_{k=1}^{n+1} A_k^s u_{n+1-k} ,  g_0 \right\rangle_{\L \left( \R_{x_2}, \C^2\right)} \left( s, \cdot \right) \mathrm{d}s. \label{eq:deff}
\end{align}
Notably, $(f_n)_{n \in \N} \in \mathscr{C}^\infty([0,T),\mathscr{S}(\R))^\N$ and $(u_n)_{n \in \N} \in \mathscr{C}^\infty([0,T), \mathscr{S}(\R^2,\C^2) \cap \ker(A_0)^\perp)^\N$.
\end{proposition}

\subsection{Smooth time-dependent Schwartz estimates}

\subsubsection{Hierarchy of estimates}
To obtain the sharpest estimates in Theorem~\ref{theo:main}, we compute the $\L$-norm of $\displaystyle{\WP_{x_t}^\varepsilon[a^{(n)}]}$.
In view of the form of each $a_k$, we proceed by induction.
To understand what is happening, let us show it with $n=2$.
According to Proposition~\ref{prop:decomposition}, the function $u_2$ is given by
\[
    u_2=-A_0^{-1}\left(A_1^t\left(u_1+\mathrm{p}_t^\dagger f_1\otimes g_0\right)+A^t_2u_0\right).
\]
The control of the $\L$-norm of $A_0^{-1}$ is done in Proposition~\ref{prop:estimateofT-1}.
Next we remark that $A_2^t$ act like an operator of multiplication by $y^3$.
So there exists a constant $C_t$ depending on time such that
\[
    \|A^t_2u_0\|_{\L} \leqslant C_t \|y^3 u_0\|_{\L}.
\]
In the same way there exists a constant $C_t$ depending on time such that
\[
    \left\| A_1^t \left(u_1+\mathrm{p}_t^\dagger f_1\otimes g_0\right) \right\|_{\L}\leqslant C_t \left\|\partial_t \left(u_1+\mathrm{p}_t^\dagger f_1\otimes g_0\right)\right\|_{\L} + \left\|y^2\left(u_1+\mathrm{p}_t^\dagger f_1\otimes g_0\right)\right\|_{\L}
\]
So we need to control the $\L$-norm of $\partial_t \left(u_1+\mathrm{p}_t^\dagger f_1\otimes g_0\right)$ and those of $y^2\left(u_1+\mathrm{p}_t^\dagger f_1\otimes g_0\right)$.
Finally, to estimate $u_2$, we need to control some time-dependent Schwartz semi-norms of $f_0, u_1$ and $f_1$.

\bigskip

In conclusion, to estimate the $\L$-norm of the function $\displaystyle{\WP_{x_t}^\varepsilon[a^{(n)}]}$, we need to compute the time-dependent Schwartz semi-norms of the family $(\partial_t^p u_n)_{n\in\N}$ and $(\partial_t^p\mathrm{p}_t^\dagger f_n)_{n\in\N}$ for all $p\in\N$.

\subsubsection{Main estimates}
Let us consider the following family of semi-norms.
For all $K \in \N$, we define \emph{$\L$-Schwartz semi-norm} of order $K$ as the following semi-norm, denoted $\|\cdot\|_{\Sigma(K)}$ and defined for all $f \in \mathscr{S}(\R)$ by
\[
    \| f \|_{\Sigma(K)} \coloneqq \underset{\alpha + \beta \leqslant K}{\sup_{(\alpha,\beta) \in \N^2}} \left\| x^\alpha \partial_x^\beta f \right\|_{\L(\R)}.
\]
We define smooth time-dependent Schwartz semi-norms (in dimension $\ell \in \{1,2\}$) by, for all $(\alpha,\beta) \in \N^\ell \times \N^\ell$, for all $p \in \N$, for all $u \in \mathscr{C}^\infty \left( \R,\mathscr{S} \left( \R^\ell,\C^\ell \right) \right)$,
\begin{equation} \label{eq:NORM}
    \| u \|_{\alpha,\beta,p} \coloneqq \left\| x^\alpha \partial_x^\beta \partial_t^p u \right\|_{\L\left( \R^\ell,\C^\ell \right)}.
\end{equation}
As we will see in Theorem~\ref{THM}, $\alpha$ and $\beta$ will not play the same role, therefore we do not consider a supremum over $\alpha + \beta \leqslant K$.
The main result of Section~\ref{sec:Schwartz} establishes control of all the smooth time-dependent Schwartz semi-norms of $(u_n)_{n\in\N}$ and $(\mathrm{p}_t^\dagger f_n)_{n\in\N}$, which are given by the following theorem.

\begin{theorem} \label{THM}
Let $(f_n)_{n \in \N} \in \mathscr{C}^\infty([0,T), \mathscr{S}(\R,\C))$ and $(u_n)_{n \in \N} \in \mathscr{C}^\infty \left( [0,T), \mathscr{S}(\R^2,\C^2) \cap \ker(A_0)^\perp \right)^\N$ satisfying~\eqref{eq:goodsystem}.
Then for all $n \in \N\setminus\{0\}$, for all $(\alpha,\beta) \in \N^2 \times \N^2$, for all $p \in \N$, there exists $C_{\alpha,\beta,p,n}$ such that,
\begin{itemize}
    \item in Setting ${\rm B}_\delta$,
\end{itemize}
\begin{align}
    \| u_n \|_{\alpha,\beta,p} & \leqslant C_{\alpha,\beta,p,n} \frac{1}{\delta^{\frac{\beta_1-\alpha_1}{2} + p + 3n - 1}} \left\| f_{\rm in} \right\|_{\Sigma(\alpha_1 + \beta_1 + 2p + 5n - 2)}, \label{eq:estimateun} \\
    \left\| \mathbf{p}_t^\dagger f_n  \right\|_{\alpha_1,\beta_1,p} & \leqslant C_{\alpha,\beta,p,n} \frac{1}{\delta^{\frac{\beta_1-\alpha_1}{2} + p + 3n}} \left\| f_{\rm in} \right\|_{\Sigma(\alpha_1 + \beta_1 + 2p + 5n)}. \label{eq:estimatefn}
\end{align}
\begin{itemize}
    \item in Settings $\mathrm{A}$, $\mathrm{B}$ and $\mathrm{C}$, for all $t\in [0,t^\star)$,
\end{itemize}
\begin{align}
    \| u_n(t,\cdot) \|_{\alpha,\beta,p} & \leqslant C_{\alpha,\beta,p,n} \frac{1}{(t^\star - t)^{M_\star\frac{\beta_1-\alpha_1}{2} + p + (3M_\star-1)(n-1) + 2M_\star}} \left\| f_{\rm in} \right\|_{\Sigma(\alpha_1 + \beta_1 + 2p + 5n - 2)}, \label{eq:estimateunM} \\
    \left\| \mathbf{p}_t^\dagger f_n(t,\cdot) \right\|_{\alpha_1,\beta_1,p} & \leqslant C_{\alpha,\beta,p,n} \frac{1}{(t^\star - t)^{M_\star\frac{\beta_1-\alpha_1}{2} + p + (3M_\star-1)n}} \left\| f_{\rm in} \right\|_{\Sigma(\alpha_1 + \beta_1 + 2p + 5n)}. \label{eq:estimatefnM}
\end{align}
\end{theorem}

This theorem implies in particular smooth time-dependent Schwartz estimates of the family $(a_k)_{k \in \N}$, according to~\eqref{eq:defauf}.
It remains to link these estimates with the one of the proximity between our wave packet approximation and the exact solution.
Such estimate is called energy estimate and is given by
\begin{equation}\label{eq:energyestimate}
    \left\|\left(\Psi-\WP_{x_t}^\varepsilon[a^{(n)}]\right)(t,\cdot)\right\|_{\L}\le \left\|\Psi_0-\WP_{x_0}^\varepsilon[a^{(n)}(0,\cdot)]\right\|_{\L}
    + \frac{1}{\ep} \int_0^t \left\| (\eps \D_s + \H) \WP^\ep_{x_s} \left[ a^{(n)} \right](s,\cdot) \right\|_{\L} \diff s.
\end{equation}
So it suffices to estimate the right hand side of \eqref{eq:energyestimate}. 
Estimates of family $(A_k^t)_{k \geqslant 0}$ and $\widetilde{R}_t^{n+1}$ lead to the following lemma.

\begin{lemma} \label{lem:an}
For all $n \in \N\setminus\{0\}$, there exists a constant $C_n > 0$ such that for all $t \in [0,T)$,
\begin{align*}
	\left\| (\eps \D_t + \H)\WP^\ep_{x_t,0} \left[ a^{(n)} \right] (t,\cdot) \right\|_{\L} \leqslant C_n & \Bigg( \ep^{\frac{n+1}{2}} \frac{\sqrt{\ep}}{r_t} \left( \left\| u_n(t,\cdot) \right\|_{\L} + \left\| x_1 \partial_1 u_n(t,\cdot) \right\|_{\L} + \left\| x_2 \partial_2 u_n(t,\cdot) \right\|_{\L} \right) \\
	& + \ep^{\frac{n+1}{2}} \sqrt{\ep} \Big( \left\| \partial_t u_n(t,\cdot) \right\|_{\L} + \left\| x_1 \partial_2 u_n(t,\cdot) \right\|_{\L} + \left\| x_2 \partial_1 u_n(t,\cdot) \right\|_{\L} \Big) \\
	& + \ep^{\frac{n}{2}} \sqrt{\frac{\ep}{r_t}} \sum_{k = 1}^{n+1} \left( \frac{\ep}{r_t} \right)^{\frac{k}{2}} \underset{|\gamma| = k+1}{\sum_{\gamma \in \N^2}} \left\| x^\gamma u_n(t,\cdot) \right\|_{\L} \\
	& + \sqrt{\frac{\ep}{r_t}} \sum_{k = 2}^{n+1} \left( \frac{\ep}{r_t} \right)^{\frac{k}{2}} \sum_{j = n+1-k}^{n-1}  \ep^\frac{j}{2} \Big( \left\| \mathbf{p}_t^\dagger f_j(t,\cdot) \right\|_{\L} + \left\| x_1^{k+1} \mathbf{p}_t^\dagger f_j(t,\cdot) \right\|_{\L} \Big).
\end{align*}
\end{lemma}

This Lemma, associated with Theorem~\ref{THM}, concludes the proof of Theorem~\ref{thm:distance}~and~\ref{theo:main}, as we will see in Section~\ref{sec:PROOF}.

\subsection{Notations}

For all operator $A$ acting on a Hilbert space $\mathcal H$, the operator $A^\dagger$ will denote the adjoint of $A$ according to the scalar product induced by $\mathcal H$.

\medskip

The operator $\mathfrak{a} \coloneqq x + \partial_x$, called the annihilation operator, acts on $\L(\R)$.
We denote by $\mathfrak{h}_0$ the first Hermite function which is defined as follows
\begin{equation} \label{eq:h0}
    \begin{array}{ccccc}
        \mathfrak{h}_0: & \R & \longrightarrow & \R \\
         & x & \longmapsto & \displaystyle{\frac{1}{\sqrt{2\pi}} e^{-x^2/2}}.
    \end{array}
\end{equation}
Inductively the $n$-$\mathrm{th}$ Hermite function $\mathfrak{h}_{n}$ is given by
\begin{equation} \label{eq:defhn}
    \mathfrak{h}_n = \frac{(\mathfrak{a}^\dagger)^n\mathfrak{h}_0}{\|(\mathfrak{a}^\dagger)^n\mathfrak{h}_0\|_{\L(\R,\C)}}.
\end{equation}

In the proofs, to simplify estimate computation, we will use the letter $C$ (for example, $a \leqslant Cb$) to denote a generic constant whose value may change from one line to another.
Moreover, if the constant $C$ depends on another quantified variable $\alpha$, we will write $a \leqslant C_\alpha b$.

\section{Wave packet expansion}\label{sec:solve-transport}

In Section~\ref{subsec:construction}, we introduce transformation and computation to obtain the system~\eqref{eq:goodsystem}.
We also explicit the spectral decomposition of the principal operator $A_0$ in Section~\ref{subsec:spectral_analysis}.
In Section~\ref{subsec:existenceuniqueness}, we construct by induction the unique (up to the initial condition) solution to~\eqref{eq:goodsystem}.

\subsection{Hierarchy of equations} \label{subsec:construction}

\subsubsection{Straightening the line} \label{sec:normal form}

\paragraph{Normal form operators.}
In this section, we prove Proposition~\ref{prop:TA} and Proposition~\ref{prop:TAA}.
This transformation can be understood as if we change the curve $\E$ into a line.
To proceed we introduce the following unitary operators acting on $\L(\R^2,\C^2)$, defined, for all $t\in [0,T)$, $\theta \in \R$, for all $u \in \L(\R^2,\C^2)$, $y \in \R^2$ by
\begin{align*}
    (\mathrm{P}_t u)(y) & \coloneqq \sqrt{r_t} u \left( \sqrt{r_t} y \right), \\
    (\mathcal{R}_{\theta} u)(y) & \coloneqq u \left( R_{\theta} y \right) = u \left( \begin{pmatrix} \cos(\theta) & \sin(\theta) \\ -\sin(\theta) & \cos(\theta) \end{pmatrix} y \right), \\
    (\mathrm{U}_\theta u)(y) & \coloneqq \frac{-i}{\sqrt2} \begin{pmatrix} e^{-i\frac{\theta}{2}} & e^{- i\frac{\theta}{2}} \\ e^{ i\frac{\theta}{2}} & -e^{i\frac{\theta}{2}} \end{pmatrix} u(y).
\end{align*}
This allow us to write $\mathscr{U}_t \coloneqq \mathcal{R}_{\theta_t} U_{\theta_t} \mathrm{P}_t$.
Moreover, they satisfy the following commutation rules,
\[
    \left[ \mathrm{U}_\theta , \mathcal{R}_\theta \right] = 0, \quad \left[ \mathrm{P}_t , \mathcal{R}_\theta \right] = 0, \quad \left[ \mathrm{U}_\theta , \mathrm{P}_t \right] = 0.
\]
Let us first conjugate the operator $T_0^t$ defined in~\eqref{eq:T_O}.

\begin{proof}[Proof of Proposition~\ref{prop:TA}]
Let $t\in [0,T)$.
A direct computation gives
\[
    \mathrm{P}_t^{-1} T_0^t \mathrm{P}_t = \sqrt{r_t}\left(-\dot{x}_t \cdot \D_y - \dot{x}_t^\perp \cdot y  \sigma_3 + \D_1\sigma_1 + \D_2 \sigma_2\right).
\]
Next we conjugate the last equation by $\mathcal{R}_{\theta_t}$ which gives
\[
    \mathcal{R}_{\theta_t}^{-1} \mathrm{P}_t^{-1} T_0^t \mathrm{P}_t \mathcal{R}_{\theta_t} = \sqrt{r_t} \left( \D_1 +  x_2 \sigma_3 + \begin{pmatrix} 0 & e^{-i\theta_t} \\ e^{i\theta_t} & 0 \end{pmatrix} \D_1 + \begin{pmatrix} 0 & -ie^{-i\theta_t} \\ ie^{i\theta_t} & 0 \end{pmatrix} \D_2 \right).
\]
Finally, we conclude with the following identities,
\[
    \mathrm U_{\theta_t}^{-1} \sigma_3 \mathrm U_{\theta_t} = \sigma_1, \quad \mathrm U_{\theta_t}^{-1} \begin{pmatrix} 0 & e^{-i\theta_t} \\ e^{i\theta_t} & 0 \end{pmatrix} \mathrm U_{\theta_t} = \sigma_3, \quad \mathrm U_{\theta_t}^{-1} \begin{pmatrix} 0 & -ie^{-i\theta_t} \\ ie^{i\theta_t} & 0 \end{pmatrix} \mathrm U_{\theta_t} = -\sigma_2.
\]
\end{proof}

Let us now conjugate the family of operators $(T_k^t)_{k \in \N\setminus\{0\}}$ defined in~\eqref{eq:T_1}-\eqref{eq:T_k} and the remainder $R_t^{n+1}$, defined in~\eqref{eq:R_t}, for $n \in \N$.

\begin{proof}[Proof of Proposition~\ref{prop:TAA}]
Let $t \in [0,T)$, $n \in \N$ and $k \in \N\setminus\{0\}$.
A direct computation gives
\begin{align*}
    \mathcal{R}_{\theta_t}^{-1} \mathrm{P}_t^{-1} R_t^{n+1} \mathrm{P}_t \mathcal{R}_{\theta_t} & = \frac{1}{r_t^{\frac{n+2}{2}}} \int_0^1 \frac{(1-u)^{n+1}}{(n+1) !} \nabla^{n+2} \m \left( x_t+u\sqrt{\eps} R_{\theta_t}^\dagger x \right).\left(R_{\theta_t}^\dagger x\right)^{n+2} \diff u, \\
    \mathcal{R}_{\theta_t}^{-1} \mathrm{P}_t^{-1} \nabla^{k+1} \m(x_t) x^{k+1} \mathrm{P}_t \mathcal{R}_{\theta_t} & = \frac{1}{r_t^{\frac{k+1}{2}}} \nabla^{k+1} \m(x_t) \left(R_{\theta_t}^\dagger x\right)^{k+1}.
\end{align*}
With the identity $\mathrm U_{\theta_t}^\dagger \sigma_3 \mathrm U_{\theta_t} = \sigma_1$, we obtain the result for $A_k^t$ with $k \geqslant 2$ and $\widetilde{R}_t^{n+1}$.
Now, it remains to describe the action of $\D_t$ on $\mathscr{U}_t$ to compute $A_1^t$.
Let $f \in \mathscr{S}(\R^2,\C^2)$, then for all $t \in [0,T)$ and for all $x \in \R^2$,
\begin{align*}
    \D_t \Big( \mathscr{U}_t[f](x) \Big) = & \mathscr{U}_t[\D_t f](x) - \frac{\dot{\theta}_t}{2\sqrt2} \begin{pmatrix} -e^{-i\frac{\theta_t}{2}} & -e^{- i\frac{\theta_t}{2}} \\ e^{ i\frac{\theta_t}{2}} & -e^{i\frac{\theta_t}{2}} \end{pmatrix} \mathcal{R}_{\theta_t} \mathrm{P}_t [f] (x) \\
    & + \mathrm{U}_{\theta_t} \mathcal{R}_{\theta_t} \mathrm{P}_t \left( -i\dot{\theta}_t \left( \begin{pmatrix} -\sin(\theta_t) & \cos(\theta_t) \\ -\cos(\theta_t) & -\sin(\theta_t) \end{pmatrix} R_{\theta_t}^{-1} x \right) \cdot \nabla f \right)(x) \\
    & + \mathrm{U}_{\theta_t} \mathcal{R}_{\theta_t} \mathrm{P}_t \left( \frac{-i\dot{r}_t}{2 r_t} \left( 1 + x \cdot \nabla \right) f \right)(x)
\end{align*}
Finally we conclude with the following identities
\[
    \begin{pmatrix} -\sin(\theta_t) & \cos(\theta_t) \\ -\cos(\theta_t) & -\sin(\theta_t) \end{pmatrix}R_{\theta_t}^{-1}x = \i\sigma_2x = -x^\perp \quad , \quad \mathrm{U}_{\theta_t}^{-1} \frac{1}{\sqrt{2}} \begin{pmatrix} -e^{-i\frac{\theta_t}{2}} & -e^{- i\frac{\theta_t}{2}} \\ e^{ i\frac{\theta_t}{2}} & -e^{i\frac{\theta_t}{2}} \end{pmatrix} = - \sigma_1.
\]
\end{proof}

\subsubsection{Spectral decomposition of the reduced leading order operator} \label{subsec:spectral_analysis}

In this section, we study the operator-valued symbol $M$ of $A_0$ defined as the Fourier transform in the $x_1$-variable by the relation $M \coloneqq \mathcal{F}_{x_1} A_0 \mathcal{F}_{x_1}^\dagger$, which is given by
\begin{equation} \label{eq:defM}
    \begin{array}{ccccc}
        M: & \R & \longrightarrow & \mathcal{B}(\L(\R,\C^2)) \\
         & \xi_1 & \longmapsto & \begin{pmatrix}
                    2\xi_1 & \mathfrak{a}\\
                    \mathfrak{a}^\dagger & 0
                    \end{pmatrix}
    \end{array}
\end{equation}
Let us first review well-known properties of $\mathfrak{a}$ and $\mathfrak{a}^\dagger$.
We refer to~\cite[Chapter 6]{Zworski} for more material on this topic.

\begin{proposition}\label{prop:propriété création anihilation}
We have the following properties.
\begin{itemize}
    \item $\mathfrak{a}\mathfrak{a}^\dagger = -\partial_x^2+x^2+1$.
    \item $\operatorname{sp}\left(\mathfrak{a}\mathfrak{a}^\dagger\right) = \lbrace2(n+1)~|~n\in\N\rbrace$.
    \item The Hermite functions $(\mathfrak{h}_n)_{n\in\N}$ form an Hilbertian basis of $\L(\R,\C)$, of eigenfunctions of $\mathfrak{a}\mathfrak{a}^\dagger$, and for all $n\in\N$, we have 
    \[
        \mathfrak{a}\mathfrak{a}^\dagger\mathfrak{h}_n = 2(n+1) \mathfrak{h}_n.
    \]
    \item Moreover $\mathfrak{a}$ and $\mathfrak{a}^\dagger$ satisfies the following identities
    \[
        \mathfrak{a}\mathfrak{h}_n=\sqrt{2n}\mathfrak{h}_{n-1},\quad \mathfrak{a}^\dagger \mathfrak{h}_n=\sqrt{2(n+1)}\mathfrak{h}_{n+1}.
    \]
\end{itemize}
\end{proposition}

The next Proposition consists in diagonalizing (at fixed $\xi$) the operator $(M(\xi),\mathcal{B}^1(\R,\C^2))$, which is self-adjoint with compact resolvent.
This decomposition of $M$ was already mentioned in~\cite{Vacelet}.
A proof of this Proposition is given at the beginning of Appendix~\ref{app:spectral theory}.

\begin{proposition}[Spectral decomposition of $M(\xi)$] \label{theo:spectrumt}
For all $\xi \in \R$, there exists an Hilbertian basis $(g_n(\xi))_{n\in\Z}$ of $\L(\R,\C^2)$ and complex numbers $(\lambda_n(\xi))_{n \in \Z}$ such that 
\begin{equation} \label{eq:decompositionwidehatA}
    M(\xi) = \sum_{n \in \Z} \lambda_n(\xi)\Pi_n(\xi) \coloneqq \sum_{n \in \Z} \lambda_n(\xi)\left(g_n(\xi)\otimes g_n(\xi)\right),
\end{equation}
where
\begin{equation} \label{eq:g0}
    \lambda_0 \coloneqq 0 \quad \text{and} \quad g_0 \coloneqq \begin{pmatrix} 0 \\ \mathfrak{h}_0 \end{pmatrix},
\end{equation}
and for all $n\in\Z\backslash\{0\}$, 
\begin{equation} \label{eq:gn}
    \lambda_n(\xi) \coloneqq \xi + \sgn(n) \sqrt{\xi^2 + 2|n|}, \quad \text{and} \quad g_n(\xi) \coloneqq \alpha_n(\xi) \displaystyle\begin{pmatrix}\mathfrak{h}_{|n|-1} \\ \frac{\sqrt{2|n|}}{\lambda_n(\xi)} \mathfrak{h}_{|n|}\end{pmatrix} 
\end{equation}
with $\displaystyle{\alpha_n (\xi) \coloneqq \frac{1}{\sqrt{2}} \sqrt{ 1 + \sgn(n) \frac{\xi}{\sqrt{\xi^2 + 2|n|}} }}$ chosen for normalization.
\end{proposition}

Notably, for all $\xi \in \R$,
\[
    \ker(M(\xi)) = \Span{(g_0)},
\]
which implies the identities~\eqref{eq:kerA}.

\subsection{Construction of the approximation} \label{subsec:existenceuniqueness}

In this section, we prove Proposition~\ref{prop:decomposition} by induction.
In Section~\ref{subsub:initialdata}, we first construct $u_0$ and $f_0$ and show that $f_0$ is uniquely determined by $f_{\rm in}$.
Then, in Section~\ref{subsub:induction}, we construct the base case of Proposition~\ref{prop:decomposition}, that is $u_1$ and $f_1$ (since it helps to understand the role of the equation in the induction proof).
Finally we construct $u_k$ and $f_k$ by induction.

\subsubsection{Initial data effects on first terms}\label{subsub:initialdata}

Let us construct $f_0$ and $u_0$ such that
\begin{equation}\label{eq:A_0}
\begin{cases}
    & A_0 u_0 = 0, \\
     &\Big( u_0 + \mathbf{p}_t^\dagger f_0 \otimes g_0 \Big) (0,\cdot) = \mathbf{p}_0^\dagger f_{\mathrm{in}} \otimes g_0.
 \end{cases}   
\end{equation}
Since $u_0$ belongs to $\mathscr{C}^\infty \left( [0,T), \mathscr{S}(\R^2,\C^2) \cap \ker(A_0)^\perp \right)$, then $u_0 = 0$.
Equation~\eqref{eq:A_0} only describes $f_0$ at time $t=0$ therefore, to entirely determine $f_0$, we also need Equation~\eqref{eq:goodsystem} for $n=1$ which is
\begin{equation} \label{eq:A_1}
\begin{cases}
    A_0 u_1 + A_1^t \Big( \mathbf{p}_t^\dagger f_0 \otimes g_0 \Big) = 0, \\
    f_0(0,\cdot) = f_{\mathrm{in}}.
\end{cases}
\end{equation}
In order to solve hierarchy of equations, Equation~\eqref{eq:A_1} needs to be well-posed, which is the case if, and only if,
\[
    A_1^t \Big( \mathbf{p}_t^\dagger f_0 \otimes g_0 \Big) \in \mathscr{C}^\infty \Big( [0,T), \mathscr{S}(\R^2,\C^2) \cap \ker(A_0)^\perp \Big).
\]
So we have to understand the action of $A_1^t$ on $\mathscr{C}^\infty \left( [0,T), \mathscr{S}(\R^2,\C^2) \cap \ker(A_0) \right)$ projected on $\ker\left( A_0 \right)$.

\begin{lemma} \label{lem:A-1onkernel}
Let $f \in \mathscr{C}^1 \left( [0,T), \mathscr{S}\left( \R,\C \right) \right)$, then, for all $\widetilde{f} \in \mathscr{C} \left( [0,T),\mathscr{C}^\infty_c \left( \R, \C \right) \right)$, for all $t \in [0,T)$,
\[
    \left\langle A_1^t \left(\mathbf{p}_t^\dagger f \otimes g_0\right) , \widetilde{f} \otimes g_0 \right\rangle_{\L \left( \R^2, \C^2\right)} = \left\langle \frac{1}{i \sqrt{r_t}} \mathbf{p}_t^\dagger \partial_t f , \widetilde{f} \right\rangle_{\L \left( \R, \C \right)}.
\]
Notably, the function $f$ does not depend on time if, and only if, for all $\widetilde{f} \in \mathscr{C} \left( [0,T),\mathscr{C}^\infty_c \left( \R, \C \right) \right)$, for all $t \in [0,T)$,
\[
    \left\langle A_1^t \left( \mathbf{p}_t^\dagger f \otimes g_0\right) , \widetilde{f} \otimes g_0 \right\rangle_{\L \left( \R^2, \C^2\right)} = 0.
\]
\end{lemma}

\begin{proof}
Let $f \in \mathscr{C}^1 \left( [0,T), \mathscr{S}\left( \R,\C \right) \right)$ and $\widetilde{f} \in \mathscr{C}\left( [0,T), \mathscr{C}^\infty_c\left( \R,\C \right) \right)$.
First, we remark that $\displaystyle{ \sigma_1 \begin{pmatrix} 0 \\ 1 \end{pmatrix} \perp \begin{pmatrix} 0 \\ 1 \end{pmatrix} }$ then for all $t\in [0,T)$,
\[
    \left\langle \left( \frac{\dot{\theta}_t}{2\sqrt{r_t}} + \frac{1}{2 r_t\sqrt{r_t}} \nabla^{2} \m(x_t) \left(R_{\theta_t}^\dagger x\right)^{2} \right) \sigma_1 \left( \mathbf{p}_t^\dagger f \otimes g_0\right) , \widetilde{f} \otimes g_0 \right\rangle_{\L \left( \R^2, \C^2\right)} = 0.
\]
Since $\displaystyle{g_0 = \mathfrak{h}_0 \begin{pmatrix} 0 \\ 1\end{pmatrix}}$, $\mathfrak{a} \mathfrak{h}_0 = 0$ and $\mathfrak{a}^\dagger \mathfrak{h}_0 = \sqrt{2}\mathfrak{h}_1$, then,
\[
    2 x_2 \mathfrak{h}_0 = -2 \partial_2 \mathfrak{h}_0 = \sqrt{2} \mathfrak{h}_1,
\]
but $\mathfrak{h}_1 \perp \mathfrak{h}_0$ then $x^\perp \cdot \nabla \left(\mathbf{p}_t^\dagger f \otimes g_0 \right) \perp \widetilde{f} \otimes g_0$, so
\begin{equation} \label{eq:A_1_on_ker}
   \left\langle A_1^t \left( \mathbf{p}_t^\dagger f \otimes g_0\right) , \widetilde{f} \otimes g_0 \right\rangle_{\L \left( \R^2, \C^2\right)} = \left\langle \frac{1}{i\sqrt{r_t}} \left( \frac{\dot{r}_t}{2r_t} (1 + x \cdot \nabla) + \partial_t \right) \left( \mathbf{p}_t^\dagger f \otimes g_0\right) , \widetilde{f} \otimes g_0 \right\rangle_{\L \left( \R^2, \C^2\right)}.
\end{equation}
By using the same identities for $\mathfrak{h}_1$, one can prove that
\[
    (1 + x_2 \partial_2) \mathfrak{h}_0 = \frac{1}{2} \mathfrak{h}_0 - \frac{1}{\sqrt{2}} \mathfrak{h}_2,
\]
then~\eqref{eq:A_1_on_ker} is reduced to
\[
    \left\langle A_1^t \left( \mathbf{p}_t^\dagger f \otimes g_0\right) , \widetilde{f} \otimes g_0 \right\rangle_{\L \left( \R^2, \C^2\right)}=\left\langle \frac{1}{i\sqrt{r_t}} \left( \frac{\dot{r}_t}{2r_t} \left( \frac{1}{2} + x_1 \partial_1 \right) + \partial_t \right) \left( \mathbf{p}_t^\dagger f \otimes g_0\right) , \widetilde{f} \otimes g_0 \right\rangle_{\L \left( \R^2, \C^2\right)}.
\]
Integrating along the $x_2$ variable gives
\[
    \left\langle A_1^t \left( \mathbf{p}_t^\dagger f \otimes g_0\right) , \widetilde{f} \otimes g_0 \right\rangle_{\L \left( \R^2, \C^2\right)} = \left\langle \frac{1}{i\sqrt{r_t}} \left( \frac{\dot{r}_t}{2r_t} \left( \frac{1}{2} + x_1 \partial_1 \right) + \partial_t \right) \left( \mathbf{p}_t^\dagger f \right) , \widetilde{f} \right\rangle_{\L \left( \R, \C \right)}.
\]
Moreover, as the operators $x_1 \partial_1$ and $\mathbf{p}_t^\dagger$ commute, for all $t \in [0,T)$, we have
\[
    \partial_t \left( \mathbf{p}_t^\dagger f \right) = \mathbf{p}_t^\dagger \left( -\frac{\dot{r}_t}{2r_t} \left( \frac{1}{2} + x_1 \partial_1 \right) + \partial_t \right) f,
\]
which concludes the proof.
\end{proof}

A direct corollary of this lemma is the following proposition.

\begin{proposition} \label{prop:a0}
The function $f_0 \in \mathscr{C}^\infty \left( [0,T), \mathscr{S}\left( \R,\C \right) \right)$ is solution of~\eqref{eq:A_0} and allows~\eqref{eq:A_1} to be well-posed if, and only if, for all $t \in [0,T)$,
\[
    f_0(t,\cdot) = f_{\mathrm{in}}.
\]
In particular $f_0$ does not depend on time.
\end{proposition}

\subsubsection{Proof by induction}\label{subsub:induction}

Let us define, for all $N \in \N\setminus\{0\}$,

\medskip

\noindent $\mathcal{P}(N)$: If the families $(f_n)_{n \in \llbracket 0,N \rrbracket}$ and $(u_n)_{n \in \llbracket 1,N \rrbracket}$ satisfy the family of equations~\eqref{eq:goodsystem}, then, for all $n \in \llbracket 1,N \rrbracket$, for all $t \in [0,T)$, these families are given by~\eqref{eq:defu} and~\eqref{eq:deff}.

\paragraph{Base case.}

In this section, we consider $u_0 = 0$ and $f_0 = f_{\mathrm{in}}$.
To entirely determined $u_1$ and $f_1$, as before, we also need to use~\eqref{eq:goodsystem} for $n=2$, which is
\begin{equation} \label{eq:A_2}
\begin{cases}
    A_0u_2 + A_1^t \Big( u_1 + \mathbf{p}_t^\dagger f_1 \otimes g_0 \Big) + A_2^t \Big( \mathbf{p}_t^\dagger f_0 \otimes g_0 \Big) = 0, \\
    f_1(0,\cdot) = 0.
\end{cases}
\end{equation}
Let us first highlight the following property which will simplify computations.

\begin{remark} \label{rem:actionAkonker} Since $\displaystyle{ \sigma_1 \begin{pmatrix} 0 \\ 1 \end{pmatrix} \perp \begin{pmatrix} 0 \\ 1 \end{pmatrix} }$ then, for all $k \geqslant 2$ and for all $t \in [0,T)$,
\[
    A_k^t \Big( \ker (A_0) \Big) \subset \ker (A_0)^\perp.
\]
\end{remark}

Let us construct $u_1 \in \mathscr{C}^\infty \left( [0,T), \mathscr{S}(\R^2,\C^2) \cap \ker(A_0)^\perp \right)$ solution of~\eqref{eq:A_1} and $f_1 \in \mathscr{C}^\infty([0,T), \mathscr{S}(\R,\C))$ such that~\eqref{eq:A_2} is well-posed.
According to Proposition~\ref{prop:a0}, by construction, for all $t \in [0,T)$,
\[
    A_1^t \left( \mathbf{p}_t^\dagger f_0 \otimes g_0 \right) \in \ker (A_0)^\perp,
\]
then, by inverting $A_0$ on $\ker (A_0)^\perp$, we obtain $u_1$ on $\ker(A_0)^\perp$, hence, we obtain~\eqref{eq:defu} for $n = 1$.
Moreover, let us suppose $f_1$ allows~\eqref{eq:A_2} to be well-posed, then
\[
    A_1^t \Big( u_1 + \mathbf{p}_t^\dagger f_1 \otimes g_0 \Big) + A_2^t \Big( \mathbf{p}_t^\dagger f_0 \otimes g_0 \Big) \in \ker \left( A_0 \right)^\perp.
\]
According to Remark~\ref{rem:actionAkonker}, $A_2^t ( \mathbf{p}_t^\dagger f_0 \otimes g_0 ) \in \ker \left( A_0 \right)^\perp$, therefore we need
\[
    A_1^t \Big( u_1 + \mathbf{p}_t^\dagger f_1 \otimes g_0 \Big) \in \ker \left( A_0 \right)^\perp.
\]
By using Lemma~\ref{lem:A-1onkernel}, for all $\widetilde{f} \in \mathscr{C} \left( [0,T),\mathscr{C}^\infty_c \left( \R, \C \right) \right)$ and for all $t \in [0,T)$,
\[
    \left\langle \frac{1}{i \sqrt{r_t}} \mathbf{p}_t^\dagger \partial_t f_1 , \widetilde{f} \right\rangle_{\L \left( \R, \C \right)} = -\left\langle A_1^t u_1 ,  \widetilde{f} \otimes g_0 \right\rangle_{\L \left( \R^2, \C^2\right)} = \left\langle A_1^t A_0^{-1} A_1^t \Big( \mathbf{p}_t^\dagger f_0 \otimes g_0 \Big) ,  \widetilde{f} \otimes g_0 \right\rangle_{\L \left( \R^2, \C^2\right)},
\]
then,
\[
    \frac{1}{i \sqrt{r_t}} \mathbf{p}_t^\dagger \partial_t f_1 = \left\langle A_1^t A_0^{-1} A_1^t \Big( \mathbf{p}_t^\dagger f_0 \otimes g_0 \Big) ,  g_0 \right\rangle_{\L \left( \R_{x_2}, \C^2\right)}.
\]
Since $f_1(0,\cdot) = 0$, then we conclude that $f_1$ satisfy~\eqref{eq:deff}.

\paragraph{Induction step.}

Let $N \in \N\setminus\{0\}$ such that $\mathcal{P}(N)$ holds.
Let us prove $\mathcal{P}(N+1)$.
Let us suppose that $(f_n)_{n \in \llbracket 0,N+1 \rrbracket}$ and $(u_n)_{n \in \llbracket 1,N+1 \rrbracket}$ satisfy the family of equations~\eqref{eq:goodsystem}, then, according to $\mathcal{P}(N)$, for all $t \in [0,T)$,
the $N$-first terms of these families are given by~\eqref{eq:defu} and~\eqref{eq:deff}.
Notably, by definition of $f_N$, for all $t \in [0,T)$,
\[
    A_1^t \Big( \mathbf{p}_t^\dagger f_N \otimes g_0 \Big) + \sum_{k=1}^{N+1} A_k^t u_{N+1-k} \in \ker(A_0)^\perp.
\]
Moreover, according to Remark~\ref{rem:actionAkonker}, for all $t \in [0,T)$,
\[
    \sum_{k=2}^{N+1} A_k^t \Big(\mathbf{p}_t^\dagger f_{N+1-k} \otimes g_0 \Big) \in \ker(A_0)^\perp.
\]
Therefore, the equation
\[
    A_0 u_{N+1} + \sum_{k=1}^{N+1} A_k^t \Big( u_{N+1-k} + \mathbf{p}_t^\dagger f_{N+1-k} \otimes g_0 \Big) = 0,
\]
is well-posed and the (only) solution (given by inverting $A_0$ on $\ker(A_0)^\perp$) is
\[
    u_{N+1} \coloneqq -A_0^{-1} \left( \sum_{k=1}^{N+1} A_k^t \Big( u_{N+1-k} + \mathbf{p}_t^\dagger f_{N+1-k} \otimes g_0 \Big) \right).
\]
Since $(f_n)_{n \in \llbracket 1,N+1 \rrbracket}$ and $(u_n)_{n \in \llbracket 1,N+1 \rrbracket}$ satisfy~\eqref{eq:goodsystem}, then the equation
\[
    A_0 u_{N+2} + \sum_{k=1}^{N+2} A_k^t \Big( u_{N+2-k} + \mathbf{p}_t^\dagger  f_{N+2-k} \otimes g_0 \Big) = 0,
\]
is well-posed,
which implies, according to Remark~\ref{rem:actionAkonker} and the relation $u_0 = 0$, that for all $t \in [0,T)$,
\[
    A_1^t \Big( \mathbf{p}_t^\dagger  f_{N+1} \otimes g_0 \Big) + \sum_{k=1}^{N+1} A_k^t u_{N+2-k} \in \ker(A_0)^\perp.
\]
Then, according to Lemma~\ref{lem:A-1onkernel}, for all $\widetilde{f} \in \mathscr{C} \left( \R,\mathscr{C}^\infty_c \left( \R, \C \right) \right)$, for all $t \in [0,T)$,
\[
    \frac{1}{i} \frac{1}{\sqrt{r_t}} \left\langle \Big( \mathbf{p}_t^\dagger \partial_t f_{N+1} \Big) (t, \cdot) , \widetilde{f} \right\rangle_{\L \left( \R, \C \right)} = - \left\langle \sum_{k=1}^{N+1} A_k^t u_{N+2-k} , \widetilde{f} \otimes g_0 \right\rangle_{\L \left( \R^2, \C^2\right)},
\]
which conclude the proof.

\section{Smooth time-dependent Schwartz estimates}\label{sec:Schwartz}

In this section, our objective is to show Theorem~\ref{THM}.
First we estimate smooth time-dependent Schwartz semi-norm of the operators $(A_k^t)_{k \in \N}$.
Then we estimate the different terms of the wave packet expansion.

\subsection{Hierarchy of operators}

\subsubsection{Leading order operator}

To estimate the Schwartz semi-norms of $A_0^{-1}$, we will use an other family of semi-norms which encode the structure of $M$.
So let us define the two families of semi-norms indexed by $K \in \N$ and $(\alpha_1,\beta_1) \in \N^2$ defined for all $f \in \mathscr{S}(\R,\C)$, for all $u \in \mathscr{S}\left(\R^2,\C^2\right)$ by
\[
    \| f \|_{\Lambda(K)} \coloneqq \left\| \left(\mathfrak{a}\mathfrak{a}^\dagger\right)^{K/2} f \right\|_{\L(\R,\C)} \quad \text{and } \quad \| u \|_{\alpha_1, \beta_1, \Lambda(K)} \coloneqq \left\| x_1^{\alpha_1}\partial_1^{\beta_1} \left(\mathfrak{a}_2\mathfrak{a}_2^\dagger\right)^{K/2} u \right\|_{\L \left( \R^2,\C^2 \right)},
\]
where $\mathfrak{a}_2$ denotes the annihilation operator on the second variable, defined by $x_2 + \partial_2$ and $\mathfrak{a}_2^\dagger$ denotes its adjoint operator.
Equivalence of the families of semi-norms $\left(\|\cdot\|_{\alpha,\beta}\right)_{(\alpha,\beta) \in \N^2}$ and $\left(\|\cdot\|_{\Lambda(K)}\right)_{K\in\N}$ is proved in  Lemma~\ref{lem:equivalencenorm}.
In particular, at fixed $(\alpha_1,\beta_1) \in \N^2$, this implies the equivalence of the families of semi-norms $\left(\|\cdot\|_{\alpha,\beta}\right)_{(\alpha_2,\beta_2) \in \N^2}$ and $\left(\|\cdot\|_{\alpha_1, \beta_1, \Lambda(K)}\right)_{K\in\N}$.

\bigskip

The following proposition establishes the Schwartz estimate of $A_0^{-1}$ on $\ker (A_0)^\perp$.

\begin{proposition} \label{prop:estimateofT-1} For all $(\alpha_1,\beta_1) \in \N^2$, for all $K \in \N$, there exists a constant $C_{\alpha_1,\beta_1,K} > 0$ such that for all $u \in \mathscr{S} \left( \R^2,\C^2 \right) \cap \ker (A_0)^\perp$,
\[
    \left\| A_0^{-1} u \right\|_{\alpha_1, \beta_1, \Lambda(K)} \leqslant C_{\alpha_1, \beta_1, K} \left( \left\| u \right\|_{\alpha_1, \beta_1, \Lambda(K-1)} + \left\| u \right\|_{\alpha_1, \beta_1+1, \Lambda(K-2)} \right).
\]
\end{proposition}

\begin{proof} Let $v \in \mathscr{S} \left( \R^2,\C^2 \right) \cap \ker (A_0)^\perp$.
Let us recall that the family $(\widetilde{g}_n)_{n \in \Z}$ defined by, 
\[
    \widetilde{g}_0 \coloneqq g_0 =  \mathfrak{h}_0 \matp{0 \\ 1}, \; \; \text{and for all $n \in \Z\setminus\{0\}$}, \; \; \widetilde{g}_n \coloneqq \frac{1}{\sqrt{2}} \matp{ \mathfrak{h}_{|n|-1} \\ \sgn(n) \mathfrak{h}_{|n|}},
\]
is an Hilbertian basis of $\L(\R,\C^2)$.
Then, according to the expression of $\ker(A_0)$ in~\eqref{eq:kerA}, there exists $(v_n)_{n \in \Z\setminus\{0\}} \in \L(\R,\C)^{\Z\setminus\{0\}}$ such that for all $x \coloneqq (x_1,x_2) \in \R^2$,
\[
    v(x) = \sum_{n \in \Z\setminus\{0\}} v_n(x_1) \widetilde{g}_n(x_2).
\]
According to Proposition~\ref{prop:propriété création anihilation}, we compute
\[
    A_0 v = \sum_{n \in \Z\setminus\{0\}} \left( \sgn(n) \sqrt{2|n|} \ v_n + \frac{1}{\i} \left( v_n' + v_{-n}' \right) \right) \widetilde{g}_n.
\]
By inverting, we conclude that the solution (in $\ker (A_0)^\perp$) to equation $A_0 v = u$ with $u \in \mathscr{S} \left( \R^2,\C^2 \right) \cap \ker (A_0)^\perp$ is given by $\displaystyle{v \coloneqq \sum_{n \in \Z\setminus\{0\}} v_n \widetilde{g}_n}$ where for all $n \in \N\setminus\{0\}$,
\begin{align*}
    v_n & \coloneqq \frac{1}{\sqrt{2n}} u_n - \frac{1}{2\i n} (u_n' - u_{-n}') \\
    v_{-n} & \coloneqq \frac{-1}{\sqrt{2n}} u_{-n} + \frac{1}{2\i n} (u_n' - u_{-n}').
\end{align*}
Notably, we have,
\[
    v = A_0^{-1} u = \sum_{n \in \Z\setminus\{0\}} \sgn(n) \left( \frac{1}{\sqrt{2|n|}} u_n - \frac{1}{2\i n} \partial_1 (u_n - u_{-n}) \right) \widetilde{g}_n.
\]
Now, to estimate the semi-norms of $v$, we consider the Hilbertian basis $\displaystyle{\matp{\mathfrak{h}_n \\ 0}_{n \in \N} \cup \matp{0 \\ \mathfrak{h}_{n+1}}_{n \in \N}}$ of the Hilbert space $\L(\R,\C^2) \cap \ker(A_0)^\perp$.
In this setting,
\begin{align*}
    u & = \sum_{n \in \N\setminus\{0\}} (u_n + u_{-n}) \matp{\mathfrak{h}_{n-1} \\ 0} + (u_n - u_{-n}) \matp{0 \\ \mathfrak{h}_n}, \\
    v & = \sum_{n \in \N\setminus\{0\}} \frac{1}{\sqrt{2n}}(u_n - u_{-n}) \matp{\mathfrak{h}_{n-1} \\ 0} + \left( \frac{1}{\sqrt{2n}}(u_n + u_{-n}) - \frac{1}{\i n}(u_n' - u_{-n}') \right) \matp{0 \\ \mathfrak{h}_n}.
\end{align*}
According to Proposition~\ref{prop:propriété création anihilation}, for all $(\alpha_1,\beta_1,K) \in \N^3$, the semi-norm $\left\| \cdot \right\|_{\alpha_1, \beta_1, \Lambda(K)}$ can be rewritten
\[
    \left\| u \right\|_{\alpha_1, \beta_1, \Lambda(K)} = \sum_{n \in \N\setminus\{0\}} (2n)^{K/2} \left\| u_n + u_{-n} \right\|_{\alpha_1, \beta_1} + (2n+2)^{K/2} \left\| u_n - u_{-n} \right\|_{\alpha_1, \beta_1},
\]
then, by triangle inequality,
\begin{align*}
    \left\| v \right\|_{\alpha_1, \beta_1, \Lambda(K)} & \leqslant \sum_{n \geqslant 1} (2n)^{\frac{K-1}{2}} \left\| u_n - u_{-n} \right\|_{\alpha_1, \beta_1} + \frac{(2n+2)^{K/2}}{\sqrt{2n}} \left\| u_n + u_{-n} \right\|_{\alpha_1, \beta_1} + \frac{(2n+2)^{K/2}}{n} \left\| u_n' - u_{-n}' \right\|_{\alpha_1, \beta_1} \\
    & \leqslant C \left\| u \right\|_{\alpha_1, \beta_1, \Lambda(K-1)} + \sum_{n \in \N\setminus\{0\}} (2n+2)^{\frac{K-2}{2}} \left\| u_n - u_{-n} \right\|_{\alpha_1, \beta_1+1} \\
    & \leqslant C \left( \left\| u \right\|_{\alpha_1, \beta_1, \Lambda(K-1)} + \left\| u \right\|_{\alpha_1, \beta_1+1, \Lambda(K-2)} \right).
\end{align*}
\end{proof}

\subsubsection{Sub-leading order operators}

In this section, we will consider the smooth time-dependent Schwartz semi-norms $\left\| \cdot \right\|_{\alpha,\beta,p}$ defined in~\eqref{eq:NORM}.
Let us first estimate the operator $A_k^t$ defined in~\eqref{eq:Atm}, for all $k \geqslant 2$, whose expression is
\[
    A_k^t = \frac{1}{ r_t^{1 + k/2}} \frac{1}{(k+1)!} \nabla^{k+1} \m(x_t) \left(R_{\theta_t}^\dagger x\right)^{k+1} \sigma_1,
\]
according to~\eqref{eq:computeA_k}.

\begin{proposition} \label{prop:estimateAk}
For all $(\alpha,\beta) \in \N^2 \times \N^2$, for all $k \geqslant 2$, for all $p \in \N$, there exists $C_{\alpha,\beta,p,k} > 0$ such that for all $u \in \mathscr{C}^\infty \left( [0,T), \mathscr{S} \left( \R^2, \C^2 \right) \right)$,
\[
	\left\| A_k^t u \right\|_{\alpha,\beta,p} \leqslant C_{\alpha,\beta,p,k} \sum_{q=0}^p \left| \partial_t^{p-q} \left( \frac{1}{r_t^{1 + k/2}} \right) \right| \underset{|\gamma| = k+1}{\sum_{\gamma \in \N^2}} \left\| x^\gamma u \right\|_{\alpha,\beta,q}.
\]
\end{proposition}

\begin{proof}
Since all the derivatives of $\m$, $\theta$ and $r$ are bounded (especially on $[0,T)$), then we conclude by applying Leibniz' formula.
\end{proof}

A similar argument leads to the estimate of the $\L$-norm of the remainder $\widetilde{R}_t^{n+1} \coloneqq r_t^{-1/2} \mathscr{U}_t^{-1} R_t^{n+1} \mathscr{U}_t$ where $R_t^{n+1}$ was defined in~\eqref{eq:R_t}.

\begin{proposition} \label{prop:estimateremainder}
For all $n \in \N$, there exists $C_n > 0$ such that, for all $v \in \mathscr{C}^\infty \left( [0,T), \mathscr{S} \left( \R^2, \C^2 \right) \right)$,
\[
    \left\| \widetilde{R}_t^{n+1} v \right\|_{\L \left( \R^2, \C^2\right)} \leqslant \frac{C_n}{r_t^{1+\frac{n+1}{2}}} \underset{|\gamma| = n+2}{\sum_{\gamma \in \N^2}} \left\| x^\gamma v \right\|_{\L \left( \R^2, \C^2\right)}.
\]
\end{proposition}

Let us now estimate $A_1^t$ defined in~\eqref{eq:Atm}, whose expression is
\[
    A_1^t = \frac{1}{i} \frac{1}{\sqrt{r_t}} \left( \frac{\dot{r}_t}{2r_t} (1 + x \cdot \nabla) + \partial_t - \dot{\theta}_t x^\perp \cdot \nabla + \frac{i\dot{\theta}_t }{2} \sigma_1 \right) + \frac{1}{2 r_t^{3/2}} \nabla^{2} \m(x_t) \left(R_{\theta_t}^\dagger x\right)^{2} \sigma_1,
\]
according to~\eqref{eq:computeA_1}.

\begin{proposition} \label{prop:estimateA1}
For all $(\alpha,\beta) \in \N^2 \times \N^2$, for all $p \in \N$, there exists a constant $C_{\alpha,\beta,p} > 0$ such that for all $u \in \mathscr{C}^\infty \left( [0,T), \mathscr{S} \left( \R^2, \C^2 \right) \right)$,
\begin{align*}
    \left\| A_1^t u \right\|_{\alpha,\beta,p} \leqslant C_{\alpha,\beta,p} & \Bigg( \sum_{q=0}^p \left| \partial_t^{p-q} \left( \frac{1}{r_t^{3/2}} \right) \right| \left( \left\| u \right\|_{\alpha,\beta,q} + \left\| x_1 \partial_1 u \right\|_{\alpha,\beta,q} + \left\| x_2 \partial_2 u \right\|_{\alpha,\beta,q} + \underset{|\gamma| = 2}{\sum_{\gamma \in \N^2}} \left\| x^\gamma u \right\|_{\alpha,\beta,q} \right) \\
    & \quad + \frac{1}{r_t^{1/2}} \left\| u \right\|_{\alpha,\beta,p+1} + \sum_{q=0}^p \left| \partial_t^{p-q} \left( \frac{1}{r_t^{1/2}} \right) \right| \left( \left\| x_1 \partial_2 u \right\|_{\alpha,\beta,q} + \left\| x_2 \partial_1 u \right\|_{\alpha,\beta,q} \right) \Bigg).
\end{align*}
\end{proposition}

\begin{proof}
Let $u \in \mathscr{C}^\infty \left( [0,T), \mathscr{S} \left( \R^2, \C^2 \right) \right)$, $(\alpha,\beta) \in \N^2 \times \N^2$ and $p \in \N$.
We will estimate each term in the expression of $A_1^t$.
First, we recognize in the last term of $A_1^t$ an expression similar to $A_k^t$ for $k \geqslant 2$.
Then, by a similar argument as in Proposition~\ref{prop:estimateAk}, one can estimate,
\[
    \left\| \frac{1}{2 r_t^{3/2}} \nabla^{2} \m(x_t) \left(R_{\theta_t}^\dagger x\right)^{2} \sigma_1 u \right\|_{\alpha,\beta,p} \leqslant C_{\alpha,\beta,p} \sum_{q=0}^p \left| \partial_t^{p-q} \left( \frac{1}{r_t^{3/2}} \right) \right| \underset{|\gamma| = 2}{\sum_{\gamma \in \N^2}} \left\| x^\gamma u \right\|_{\alpha,\beta,q}.
\]
By applying Leibniz' formula and re-ordering the terms, we have
\[
    \left\| \frac{\dot{\theta}_t}{\sqrt{r_t}} \left( - x^\perp \cdot \nabla + \frac{i}{2} \sigma_1 \right) u \right\|_{\alpha,\beta,p} \leqslant C_{\alpha,\beta,p} \sum_{q=0}^p \left| \partial_t^{p-q} \left( \frac{1}{r_t^{1/2}} \right) \right| \left( \left\| u \right\|_{\alpha,\beta,q} + \left\| x_1 \partial_2 u \right\|_{\alpha,\beta,q} + \left\| x_2 \partial_1 u \right\|_{\alpha,\beta,q} \right).
\]
The term $\partial_t$ in $A_1^t$ gives, by re-ordering the terms,
\[
    \left\| \frac{1}{\sqrt{r_t}}\partial_t u \right\|_{\alpha,\beta,p} = \sum_{q=0}^p \left| \partial_t^{p-q} \left( \frac{1}{r_t^{1/2}} \right) \right| \left\| \partial_t u \right\|_{\alpha,\beta,q}.
\]
According to Leibniz' formula, and because all the derivatives of $r$ are bounded, then,
\[
    \left\| \frac{\dot{r}_t}{2r_t^{3/2}} (1 + x \cdot \nabla) u \right\|_{\alpha,\beta,p} \leqslant C_p \sum_{q=0}^p \left| \partial_t^{p-q} \left( \frac{1}{r_t^{3/2}} \right) \right| \left( \left\| u \right\|_{\alpha,\beta,q} + \left\| x_1 \partial_1 u \right\|_{\alpha,\beta,q} + \left\| x_2 \partial_2 u \right\|_{\alpha,\beta,q} \right),
\]
which concludes the proof.
\end{proof}

\subsubsection{Improvement on \texorpdfstring{$\ker(A_0)$}{}}

In this section, we improve Proposition~\ref{prop:estimateA1} when $A_1^t$ acts on $\mathscr{C}^1(\R, \ker\left( A_0 \right)^\perp )$ projected on $\ker\left( A_0 \right)$.

\begin{proposition} \label{prop:A-1onperp}
For all $(\alpha_1,\beta_1) \in \N^2$, for all $p \in \N$, there exists a constant $C_{\alpha_1,\beta_1,p} > 0$ such that, for all $u \in \mathscr{C}^1\left( [0,T), \mathscr{S}(\R^2,\C^2) \cap \ker\left( A_0 \right)^\perp \right)$, for all $t \in [0,T)$,
\begin{align*}
    \left\| \left\langle A_1^t u , g_0 \right\rangle_{\L \left( \R_{x_2}, \C^2\right)} \right\|_{\alpha_1,\beta_1,p} \leqslant C_{\alpha_1,\beta_1,p} & \Bigg( \sum_{q=0}^p \left| \partial_t^{p-q} \left( \frac{1}{r_t^{3/2}} \right) \right| \left( \left\| u \right\|_{\alpha,\beta,q} + \left\| x_1^2 u \right\|_{\alpha,\beta,q} \right) \\
    & \quad + \sum_{q = 0}^p \left| \partial_t^{p-q} \left( \frac{1}{r_t^{1/2}} \right) \right| \left( \left\| x_1 u \right\|_{\alpha,\beta,q} + \left\| \partial_1 u \right\|_{\alpha,\beta,q} \right) \Bigg),
\end{align*}
with $\alpha \coloneqq (\alpha_1,0)$ and $\beta \coloneqq (\beta_1,0)$.
\end{proposition}

\begin{proof} Let $u \in \mathscr{C}^1\left( [0,T), \mathscr{S}(\R^2,\C^2) \cap \ker\left( A_0 \right)^\perp \right)$, $(\alpha_1,\beta_1) \in \N^2$ and $p \in \N$.
Let us denote $\alpha \coloneqq (\alpha_1,0)$ and $\beta \coloneqq (\beta_1,0)$.
Notably, for all $t \in [0,T)$, for almost every $x_1 \in \R$,
\[
    \left\langle u(t,x_1, \cdot) , g_0 \right\rangle_{\L \left( \R_{x_2}, \C^2 \right)} = 0,
\]
then, for almost every $x_1 \in \R$,
\[
    \left\langle \left( \frac{1}{i\sqrt{r_t}} \left( \frac{\dot{r}_t}{2r_t} (1 + x_1 \partial_1) + \partial_t \right) \right) u , g_0 \right\rangle_{\L \left( \R, \C^2\right)} = 0,
\]
which allows us to compute
\[
    \left\langle A_1^t u , g_0 \right\rangle_{\L \left( \R^2, \C^2\right)} = \frac{1}{\sqrt{r_t}} \left\langle \left( \frac{1}{i} \left( \frac{\dot{r}_t}{2r_t} x_2 \partial_2 - \dot{\theta}_t x^\perp \cdot \nabla + \frac{i\dot{\theta}_t }{2} \sigma_1 \right) + \frac{1}{2 r_t} \nabla^{2} \m(x_t) \left(R_{\theta_t}^\dagger x\right)^{2} \sigma_1 \right) u , g_0 \right\rangle_{\L \left( \R, \C^2\right)}.
\]
To conclude the proof, it suffices to estimate the four terms separately.
For the first term, by passing to the adjoint,
\[
    \left\langle -ix_2 \partial_2 \frac{\dot{r}_t}{2r_t\sqrt{r_t}} u , g_0 \right\rangle_{\L \left( \R, \C^2\right)} = \left\langle \frac{\dot{r}_t}{2r_t\sqrt{r_t}} u , -ix_2 \partial_2 g_0 \right\rangle_{\L \left( \R, \C^2\right)},
\]
then, according to Cauchy-Schwarz inequality and Leibniz's formula,
\[
    \left\| \left\langle -ix_2 \partial_2 \frac{\dot{r}_t}{2r_t\sqrt{r_t}} u , g_0 \right\rangle_{\L \left( \R, \C^2\right)} \right\|_{\alpha_1,\beta_1,p} \leqslant C_p \sum_{q = 0}^p \left| \partial_t^{p-q} \left( \frac{1}{r_t^{3/2}} \right) \right| \left\| u \right\|_{\alpha,\beta,q}.
\]
For the second term, by passing to the adjoint,
\[
    \left\langle \frac{1}{\sqrt{r_t}} \dot{\theta}_t x^\perp \cdot \nabla u , g_0 \right\rangle_{\L \left( \R, \C^2\right)} = \left\langle \frac{1}{\sqrt{r_t}} \dot{\theta}_t x_1 u , \partial_2 g_0 \right\rangle_{\L \left( \R, \C^2\right)} - \left\langle \frac{1}{\sqrt{r_t}} \dot{\theta}_t \partial_1 u , x_2 g_0 \right\rangle_{\L \left( \R, \C^2\right)},
\]
then, according to Cauchy-Schwarz inequality and Leibniz's formula,
\[
    \left\| \left\langle \frac{1}{\sqrt{r_t}} \dot{\theta}_t x^\perp \cdot \nabla u , g_0 \right\rangle_{\L \left( \R, \C^2\right)} \right\|_{\alpha_1,\beta_1,p} \leqslant C \sum_{q = 0}^p \left| \partial_t^{p-q} \left( \frac{1}{r_t^{1/2}} \right) \right| \left( \left\| x_1 u \right\|_{\alpha,\beta,q} + \left\| \partial_1 u \right\|_{\alpha,\beta,q} \right).
\]
For the third term, since $t \mapsto \theta_t$ is smooth on $\R$, Leibniz' formula implies
\[
    \left\| \left\langle \frac{1}{\sqrt{r_t}} \frac{\dot{\theta}_t}{2} \sigma_1 u , g_0 \right\rangle_{\L \left( \R, \C^2\right)} \right\|_{\alpha_1,\beta_1,p} \leqslant C \sum_{q = 0}^p \left| \partial_t^{p-q} \left( \frac{1}{r_t^{1/2}} \right) \right| \left\| u \right\|_{\alpha,\beta,q},
\]
which is negligible compared to the first term.
For the fourth term, we use Cauchy-Schwarz inequality
\[
    \left\| \left\langle \frac{1}{2 r_t^{3/2}} \nabla^{2} \m(x_t) \left(R_{\theta_t}^\dagger x\right)^{2} \sigma_1 u , g_0 \right\rangle_{\L \left( \R, \C^2\right)} \right\|_{\alpha_1,\beta_1,p} \leqslant C \sum_{q=0}^p \left| \partial_t^{p-q} \left( \frac{1}{r_t^{3/2}} \right) \right| \sum_{j=0}^2 \left\| x_1^j u \right\|_{\alpha,\beta,q},
\]
which concludes the proof.
\end{proof}

\subsection{One-dimensional estimates}

In this section, we give a precise estimate of the smooth time-dependent Schwartz semi-norms of $\mathbf{p}_t f$ and of $\mathbf{p}_t^\dagger f$ with $f \in \mathscr{C}^\infty \left( [0,T), \mathscr{S} \left( \R, \C^2 \right) \right)$, where $T>0$ is such that $r_t\neq 0$ for all $t\in [0,T)$.

\begin{lemma}\label{lem:estimatept}
For all $p \in \N$, for all $(\alpha_1,\beta_1) \in \N^2$, there exists a constant $C_{\alpha_1,\beta_1,p} > 0$ such that for all $f \in \mathscr{C}^\infty \left( [0,T), \mathscr{S} \left( \R,\C^2 \right) \right)$ and for all $t \in [0,T)$,
	\begin{align*}
		\| \mathbf{p}_t f \|_{\alpha_1, \beta_1,p} & \leqslant C_{\alpha_1, \beta_1, p} \frac{1}{r_t^{\frac{\alpha_1-\beta_1}{2}}} \sum_{q=0}^p \frac{1}{r_t^{p-q}} \sum_{\ell = 0}^{p-q} \left\| f \right\|_{\alpha_1+\ell,\beta_1+\ell,q}, \\
		\left\| \mathbf{p}_t^\dagger f \right\|_{\alpha_1, \beta_1,p} & \leqslant C_{\alpha_1, \beta_1, p} \frac{1}{r_t^{\frac{\beta_1-\alpha_1}{2}}} \sum_{q=0}^p \frac{1}{r_t^{p-q}} \sum_{\ell = 0}^{p-q} \left\| f \right\|_{\alpha_1+\ell,\beta_1+\ell,q}.
	\end{align*}
\end{lemma}

Let us prove the result for $\mathbf{p}_t$ since the proof is very similar for $\mathbf{p}_t^\dagger$.

\begin{proof}
A consequence of the dilation character of $\mathbf{p}_t$ is that, for all $p \in \N$, for all $f \in \mathscr{C}^\infty \left( [0,T), \mathscr{S} \left( \R, \C^2 \right) \right)$,
\begin{equation} \label{eq:derivativeP}
    \partial_t^p \left( \mathbf{p}_t f \right)  = \left( \mathbf{p}_t Q_{t}^p f \right), 
\end{equation}
with,
\begin{equation} \label{eq:Qeta}
    Q_{t} \coloneqq  \frac{\dot{r}_t}{2r_t} \left( \frac{1}{2} + x \partial_x \right) + \partial_t,
\end{equation}
As the operator $Q_{t}^p$ is a polynomial operator of degree $p$, then it is the sum of operators of the following form,
\[
    \frac{C_{q,k,s}(t)}{r_t^q}  \left( \frac{1}{2} + x \partial_x \right)^k \partial_t^s
\]
where $C_{q,k,s}$ are bounded functions depending only of the time variable, $k \leqslant q$ and $q+s \leqslant p$.
Moreover, the operator $\left( \frac{1}{2} + x \partial_x \right)^k$ can be expanded to a sum of terms of the form $C_\ell \left( x \partial_x \right)^\ell$, with $\ell \leqslant k$, which can also be expand into sum of $x^\ell \partial_x^\ell$ for $\ell \leqslant k$.
Thus as for all $(\alpha_1,\beta_1)\in\N^2$, there exists $(C_j)_{j\in\llbracket 0,\ell\rrbracket}$ such that
\[
    x^{\alpha_1} \partial_x^{\beta_1} x^\ell \partial_x^\ell = \sum_{j=0}^\ell C_j x^{\alpha_1+j} \partial_x^{\beta_1 + j},
\] 
we deduce that for all $t\in [0,T)$,
\[
    \left\| Q_{t}^p f \right\|_{\alpha_1, \beta_1,0} \leqslant C_{\alpha_1, \beta_1, p} \sum_{q=0}^p \left( \frac{1}{r_t^{p-q}} \right) \sum_{\ell = 0}^{p-q} \left\| f \right\|_{\alpha_1+\ell,\beta_1+\ell,q}.
\]	
Finally in view of~\eqref{eq:derivativeP} and since $(\mathbf{p}_t)_{t\in [0,T)}$ is a family of unitary operators, then for all $(\alpha_1,\beta_1) \in \N^2$,
\[
    \| \mathbf{p}_t f \|_{\alpha_1,\beta_1,p} = \| \mathbf{p}_t (Q_{t}^pf) \|_{\alpha_1,\beta_1} = \left\| r_t^{\frac{\beta_1-\alpha_1}{2}} \mathbf{p}_t x^{\alpha_1} \partial_x^{\beta_1} (Q_{t}^p f)\right\|_{\L} = r_t^{\frac{\beta_1-\alpha_1}{2}} \|Q_{t}^p f\|_{\alpha_1,\beta_1}.
\]
The proof finally follows by the triangle inequality.
\end{proof}

\subsection{Wave packets expansion}

In this section, we prove Theorem~\ref{THM}.
In the first section, we write $f_n$ in an appropriate shape in order to simplify its estimate.
The next two sections are devoted to the computation of these estimates by induction.

\subsubsection{The integration process}

In order to estimate the family $(f_n)_{n \in \N}$, we estimate the family of operators $(\Phi_t)_{[0,T)}$ defined for all $t\in [0,T)$ by 
\[
    \begin{array}{ccccc}
        \Phi_t : & \mathscr{C} \left( [0,T), \L (\R,\C) \right) &\longrightarrow &  \L (\R,\C) \\
         & f & \longmapsto & \Phi_t(f)
    \end{array}
\]
such that, for all $t \in [0,T)$, for all $f \in \mathscr{C} \left( [0,T), \L (\R,\C) \right)$, for all $x_1 \in \R$,
\begin{align*}
    \Phi_t(f)(x_1) & \coloneqq  \int_0^t \sqrt{r_s} \mathbf{p}_s f \left( s, x_1 \right) \mathrm{d}s.
\end{align*}
We also introduce the following family of operators $(\varphi_k)_{k \in \N\setminus\{0\}}$, defined, for all $k \in \N\setminus\{0\}$, by
\[
    \begin{array}{ccccc}
        \varphi_k : & \mathscr{C}^1 \left( [0,T), \mathscr{S} \left( \R^2, \C^2 \right) \right) & \longrightarrow &  \mathscr{C}^1 \left( [0,T), \mathscr{S} (\R,\C) \right) \\
         & u & \longmapsto & \left\langle A_k^t u ,  g_0 \right\rangle_{\L \left( \R_{x_2}, \C^2\right)}
    \end{array}
\]
then, the family $(f_n)_{n \in \N\setminus\{0\}}$ can be rewritten as follows, for all $n \in \N\setminus\{0\}$, for all $t \in [0,T)$,
\begin{equation} \label{eq:decomposef}
    f_n(t,\cdot) = \frac{1}{i} \sum_{k=1}^n \Phi_t \left( \varphi_k u_{n+1-k} \right).
\end{equation}

\subsubsection{Analysis of the integration process}

We now investigate estimates of $(f_k)$ defined in~\eqref{eq:deff}.
Let us compute the smooth time-dependent Schwartz estimates of $\mathbf{p}_t^\dagger \Phi_t$ and $(\varphi_k)_{k \in \N\setminus\{0\}}$.

\medskip

In the setting of Theorem~\ref{thm:distance}, where the distance between $\E_1$ and $\E_2$ is strictly positive (denoted $\delta$), the function $t \mapsto r_t$ is bounded from below by $\delta$.
In settings of Theorem~\ref{theo:main}, we have $\delta=0$ so $r_{t^\star}=0$.
In particular, we may treat all these cases in the same time, by setting for all $t\in [0,T)$,
\begin{equation}\label{eq:expression r_t}
    r_t=F(t)\mathds{1}_{\lbrace\delta>0\rbrace}+ C|t-t^\star|^{M^\star}\mathds{1}_{\lbrace\delta=0\rbrace},
\end{equation}	
with $C>0$ and $T = t^\star >0$ in settings of Theorem~\ref{theo:main}.

\medskip

Let us first estimate smooth time-dependent Schwartz semi-norm of $\mathbf{p}_t^\dagger \Phi_t$.

\begin{lemma} \label{lem:estimatephi} For all $(\alpha_1,\beta_1,p) \in \N^3$, there exists $C_{\alpha_1,\beta_1,p}$ such that for all $f \in \mathscr{C}^\infty \left( [0,T), \mathscr{S} (\R,\C) \right)$ and for all $t \in [0,T)$,
\begin{align}
    \left\| \mathbf{p}_t^\dagger \Phi_t(f) \right\|_{\alpha_1,\beta_1,p} \leqslant C_{\alpha_1,\beta_1,p} & \Bigg( \frac{1}{r_t^{\frac{\beta_1 - \alpha_1}{2}}} \left( \frac{\mathds{1}_{\lbrace\delta>0\rbrace}}{F(t)^p} + \frac{\mathds{1}_{\lbrace \delta =0\rbrace}}{(t^\star-t)^p} \right) \int_0^t r_s^{\frac{\beta_1 - \alpha_1 + 1}{2}} \sum_{q=0}^p \left\| f(s,\cdot) \right\|_{\alpha_1+q,\beta_1+q} \diff s \label{eq:estimatePhi} \\
    & \quad + \sum_{q = 0}^{p-1} \left( \frac{\mathds{1}_{\lbrace\delta>0\rbrace}}{F(t)^{p-q-3/2}} + \frac{\mathds{1}_{\lbrace \delta =0\rbrace}}{(t^\star-t)^{p-q-1-M_\star/2}} \right) \sum_{\ell = 0}^{p-1-q} \left\| f \right\|_{\alpha_1+\ell,\beta_1+\ell,q} \Bigg). \nonumber
\end{align}
\end{lemma}

\begin{proof}
Let us prove the result by induction.
Let us define, for all $p \in \N$,

\bigskip

\noindent $\mathcal{P}(p)$: For all $(\alpha_1,\beta_1) \in \N^2$, there exists $C_{\alpha_1,\beta_1}$ such that for all $f \in \mathscr{C}^\infty \left( [0,T), \mathscr{S} (\R,\C) \right)$ and for all $t \in [0,T)$, inequality~\eqref{eq:estimatePhi} holds.

\medskip

\noi \textbf{Base case.} For $p = 0$, let $(\alpha_1,\beta_1) \in \N^2$, $f \in \mathscr{C}^\infty \left( [0,T), \mathscr{S} (\R,\C) \right)$ and $t \in [0,T)$, then, according to Lemma~\ref{lem:estimatept} and triangle inequality through the integration,
\begin{align*}
    \left\| \mathbf{p}_t^\dagger \Phi_t(f) \right\|_{\alpha_1,\beta_1} & = \frac{1}{r_t^{\frac{\beta_1 - \alpha_1}{2}}}  \int_0^t \left\| r_s^\frac{1}{2} \mathbf{p}_s f \left( s, \cdot \right) \right\|_{\alpha_1,\beta_1} \mathrm{d}s\\
    & \leqslant  \frac{1}{r_t^{\frac{\beta_1 - \alpha_1}{2}}} \int_0^t r_s^{\frac{\beta_1 - \alpha_1 + 1}{2}} \left\| f \left( s, \cdot \right) \right\|_{\alpha_1,\beta_1} \mathrm{d}s,
\end{align*}
so the base case holds.

\medskip

\noi \textbf{Induction step.}
Let $p \in \N$ such that $\mathcal{P}(p)$ holds.
Let us prove $\mathcal{P}(p+1)$.

Let $(\alpha_1,\beta_1) \in \N^2$, $f \in \mathscr{C}^\infty \left( [0,T), \mathscr{S} (\R,\C) \right)$ and $t \in [0,T)$.
Since~\eqref{eq:derivativeP} holds and $x_1 \partial_1$ commutes with integration over time and the operator $\mathbf{p}_s$ then
\[
    \partial_t \left( \mathbf{p}_t^\dagger \Phi_t(f) \right) = \sqrt{r_t} f(t,\cdot) + \mathbf{p}_t^\dagger\Phi_t\left(\frac{\dot{r}_t}{2r_t} \left( \frac{1}{2} + x_1 \partial_1 \right)f \right).
\]
According to triangle inequality,
\[
    \left\| \mathbf{p}_t^\dagger \Phi_t(f) \right\|_{\alpha_1,\beta_1,p+1} \leqslant C \left( \left\| \sqrt{r_t} f(t, \cdot) \right\|_{\alpha_1,\beta_1,p} + \left\| \mathbf{p}_t^\dagger \Phi_t \left( \frac{\dot{r}_t}{r_t} f \right) \right\|_{\alpha_1,\beta_1,p} + \left\| \mathbf{p}_t^\dagger \Phi_t \left( \frac{\dot{r}_t}{r_t} f \right) \right\|_{\alpha_1+1,\beta_1+1,p} \right).
\]
By Leibniz' formula,
\[
    \left\| \sqrt{r_t} f(t, \cdot) \right\|_{\alpha_1,\beta_1,p} \leqslant C_{\alpha_1,\beta_1,p} \sum_{q=0}^p \left| \partial_t^{p-q} \left( r_t^{1/2} \right) \right| \left\| f \right\|_{\alpha_1,\beta_1,q},
\]
but, for all $q \in \N$, there exists $C_q$ a constant such that for all $t \in [0, t^\star)$,
\[
    \left| \partial_t^q \left( r_t^{1/2} \right) \right| \leqslant C_q \left( \frac{\mathds{1}_{\lbrace\delta>0\rbrace}}{F(t)^{q-1/2}} + \frac{\mathds{1}_{\lbrace\delta=0\rbrace}}{(t^\star-t)^{q-M_\star/2}} \right).
\]
Then, by applying $\mathcal{P}(p)$ to the other terms, we obtain, for $\eta \in \{0,1\}$,
\begin{align*}
    & \left\| \mathbf{p}_t^\dagger \Phi_t \left( \frac{\dot{r}_t}{r_t} f \right) \right\|_{\alpha_1+\eta,\beta_1+\eta,p} \\
    & \quad \leqslant C_{\alpha_1,\beta_1,p} \Bigg( \frac{1}{r_t^{\frac{\beta_1 - \alpha_1}{2}}} \left( \frac{\mathds{1}_{\lbrace\delta>0\rbrace}}{F(t)^p} + \frac{\mathds{1}_{\lbrace\delta=0\rbrace}}{(t^\star-t)^p} \right) \int_0^t r_s^{\frac{\beta_1 - \alpha_1 + 1}{2}} \sum_{q=0}^p \frac{\dot{r}_t}{r_t} \left\| f(s,\cdot) \right\|_{\alpha_1+q+\eta,\beta_1+q+\eta} \diff s \\
    & \quad\quad\quad\quad\quad\quad \quad + \sum_{q = 0}^{p-1} \left( \frac{\mathds{1}_{\lbrace\delta>0\rbrace}}{F(t)^{p-q-3/2}} + \frac{\mathds{1}_{\lbrace\delta=0\rbrace}}{(t^\star-t)^{p-q-1-M_\star/2}} \right) \sum_{\ell = 0}^{p-1-q} \left\| \frac{\dot{r}_t}{r_t} f \right\|_{\alpha_1+\ell+\eta,\beta_1+\ell+\eta,q} \Bigg).
\end{align*}
We conclude the proof by applying Leibniz' formula.
\end{proof}

On the other hand, smooth time-dependent Schwartz estimates of the family $(\varphi_k)_{k \in \N\setminus\{0\}}$ are a consequence of Proposition~\ref{prop:estimateAk} and is exactly Proposition~\ref{prop:A-1onperp} for $k=1$.

\begin{lemma} \label{lem:estimatevarphi} Let $(\alpha_1,\beta_1,p) \in \N^3$ and let us denote $\alpha \coloneqq (\alpha_1,0)$, $\beta \coloneqq (\beta_1,0)$.
We have the following estimates.
\begin{itemize}
    \item There exists $C_{\alpha_1,\beta_1,p} > 0$ such that for all $u \in \mathscr{C}^\infty \left( [0,T), \mathscr{S} \left( \R^2, \C^2 \right) \right)$ and for all $t \in [0,T)$,
	\begin{align*}
		\left\| \varphi_1 u \right\|_{\alpha_1,\beta_1,p} \leqslant C_{\alpha_1,\beta_1,p} & \Bigg( \sum_{q=0}^p \left| \partial_t^{p-q} \left( \frac{1}{r_t^{3/2}} \right) \right| \left( \left\| u \right\|_{\alpha,\beta,q} + \left\| x_1^2 u \right\|_{\alpha,\beta,q} \right) \\
        & \quad + \sum_{q = 0}^p \left| \partial_t^{p-q} \left( \frac{1}{r_t^{1/2}} \right) \right| \left( \left\| x_1 u \right\|_{\alpha,\beta,q} + \left\| \partial_1 u \right\|_{\alpha,\beta,q} \right) \Bigg),
    \end{align*}
    with $\alpha \coloneqq (\alpha_1,0)$ and $\beta \coloneqq (\beta_1,0)$.
    \item For all integer $k \geqslant 2$,  there exists $C_{\alpha_1,\beta_1,p,k} > 0$ such that for all $u \in \mathscr{C}^\infty \left( [0,T), \mathscr{S} \left( \R^2, \C^2 \right) \right)$ and for all $t \in [0,T)$,
    \[
        \left\| \varphi_k u \right\|_{\alpha_1,\beta_1,p} \leqslant C_{\alpha,\beta,p,k} \sum_{q=0}^p \left| \partial_t^{p-q} \left( \frac{1}{r_t^{1 + k/2}} \right) \right| \left( \left\| u \right\|_{\alpha,\beta,q} + \left\| x_1^{k+1} u \right\|_{\alpha,\beta,q} \right).
    \]
\end{itemize}
\end{lemma}

Finally, it remains to precise estimates of $r_t$ and its derivatives.

\begin{remark} \label{rem:behavior_r_t}
According to the behavior of $r_t$ on $[0,T)$ described in~\eqref{eq:expression r_t}, for all $p \in \N$, for all $n \in \mathbb{Q}\setminus\{0\}$, there exists a constant $C_p$ such that for all $t \in [0,T)$,
\[
    \left| \partial_t^p \left( \frac{1}{r_t^n} \right) \right| \leqslant C_p \frac{1}{r_t^n} \left( \frac{\mathds{1}_{\lbrace\delta>0\rbrace}}{\delta^p} + \frac{\mathds{1}_{\lbrace\delta=0\rbrace}}{ (t^\star-t)^p} \right).
\]
\end{remark}

\subsection{Proof of Theorem~\ref{THM}}

We now prove by induction that~\eqref{eq:estimateun}--\eqref{eq:estimatefnM} hold for all $n \in \N\setminus\{0\}$.
Let us define, for all $n \in \N\setminus\{0\}$,

\bigskip

\noindent $\mathcal{P}(n)$: For all $(\alpha, \beta) \in \N^2 \times \N^2$, for all $p \in \N$, there exists $C_{\alpha,\beta,p} > 0$ such that~\eqref{eq:estimateun}--\eqref{eq:estimatefnM} hold.

\subsubsection{Base case} \label{sec:basecase2}

In this section, we prove the base case ($n=1$).
We derive from Lemma~\ref{lem:estimatept} and Remark~\ref{rem:behavior_r_t} that, for all $(\alpha_1,\beta_1,p) \in \N^3$, there exists $C_{\alpha_1,\beta_1,p}$ such that
\begin{equation}
    \left\| \mathbf{p}_t^\dagger f_0 \right\|_{\alpha_1,\beta_1,p} \leqslant C_{\alpha_1, \beta_1, p} \frac{1}{r_t^{\frac{\beta_1-\alpha_1}{2}}} \left( \frac{\mathds{1}_{\lbrace\delta>0\rbrace}}{\delta^p} + \frac{\mathds{1}_{\lbrace\delta=0\rbrace}}{(t^\star-t)^p} \right) \| f_{\mathrm{in}} \|_{\Sigma(\alpha_1+\beta_1+2p)}. \label{lem:estimatef0}
\end{equation}
We can then estimate smooth time-dependent Schwartz semi-norms of $u_1$.
Let $(\alpha,\beta) \in \N^2 \times \N^2$ and $p \in \N$.
Let us denote $K \coloneqq \alpha_2 + \beta_2$.
Since $A_0$ does not depend on $t$, according to Proposition~\ref{prop:estimateofT-1},
\[
    \left\| u_1 \right\|_{\alpha,\beta,p} \leqslant C_{\alpha,\beta,p} \left( \left\| \partial_t^p A_1^t \Big( \mathbf{p}_t^\dagger f_0 \otimes g_0 \Big) \right\|_{\alpha_1, \beta_1, \Lambda(K-1)} + \left\| \partial_t^p A_1^t \Big( \mathbf{p}_t^\dagger f_0 \otimes g_0 \Big) \right\|_{\alpha_1, \beta_1+1, \Lambda(K-2)} \right).
\]
Let $\widetilde{\alpha} \coloneqq (\alpha_1,\widetilde{\alpha_2})$ and $\widetilde{\beta} \coloneqq (\widetilde{\beta_1},\widetilde{\beta_2})$ with $(\widetilde{\alpha_2},\widetilde{\beta_1},\widetilde{\beta_2}) \in \N^3$.
As estimate in the proof of Proposition~\ref{prop:estimateA1}, by integrating over $x_2$, we obtain
\begin{align*}
    \left\| A_1^t \mathbf{p}_t^\dagger f_0 \otimes g_0 \right\|_{\widetilde{\alpha},\widetilde{\beta},p} \leqslant C_{\widetilde{\alpha},\widetilde{\beta},p} & \Bigg( \sum_{q=0}^p \left| \partial_t^{p-q} \left( \frac{1}{r_t^{3/2}} \right) \right| \left( \left\| \mathbf{p}_t^\dagger f_0 \right\|_{\alpha_1,\widetilde{\beta_1},q} + \left\| x_1 \partial_1 \mathbf{p}_t^\dagger f_0 \right\|_{\alpha_1,\widetilde{\beta_1},q} + \left\| x_1^2 \mathbf{p}_t^\dagger f_0 \right\|_{\alpha_1,\widetilde{\beta_1},q} \right) \\
    & \quad + \frac{1}{r_t^{1/2}} \left\| \mathbf{p}_t^\dagger f_0 \right\|_{\alpha_1,\widetilde{\beta_1},p+1} + \sum_{q=0}^p \left| \partial_t^{p-q} \left( \frac{1}{r_t^{1/2}} \right) \right| \left\| \partial_1 \mathbf{p}_t^\dagger f_0 \right\|_{\alpha_1,\widetilde{\beta_1},q} \Bigg).
\end{align*}
We conclude thanks to~\eqref{lem:estimatef0} and Remark~\ref{rem:behavior_r_t} that~\eqref{eq:estimateun} and~\eqref{eq:estimateunM} hold for $n=1$.
Finally, we estimate the smooth time-dependent Schwartz semi-norms of $\mathbf{p}_t^\dagger f_1$.
Let $(\alpha_1,\beta_1) \in \N^2$, $p \in \N$.
According to the decomposition of $f_1$ in~\eqref{eq:decomposef} and to Lemma~\ref{lem:estimatephi},
\begin{align*}
    \left\| \mathbf{p}_t^\dagger f_1 \right\|_{\alpha_1,\beta_1,p} \leqslant C_{\alpha_1,\beta_1,p} & \Bigg( \frac{1}{r_t^{\frac{\beta_1 - \alpha_1}{2}}} \left( \frac{\mathds{1}_{\lbrace\delta>0\rbrace}}{\delta^p} + \frac{\mathds{1}_{\lbrace\delta=0\rbrace}}{(t^\star-t)^p} \right) \int_0^t r_s^{\frac{\beta_1 - \alpha_1 + 1}{2}} \sum_{q=0}^p \left\| \varphi_1 u_1(s,\cdot) \right\|_{\alpha_1+q,\beta_1+q} \diff s \\
    & \quad + \sum_{q = 0}^{p-1} \left( \frac{\mathds{1}_{\lbrace\delta>0\rbrace}}{\delta^{p-q-3/2}} + \frac{\mathds{1}_{\lbrace\delta=0\rbrace}}{(t^\star-t)^{p-q-1-M_\star/2}} \right) \sum_{\ell = 0}^{p-1-q} \left\| \varphi_1 u_1 \right\|_{\alpha_1+\ell,\beta_1+\ell,q} \Bigg).
\end{align*}
Let us denote for all $\ell \in \N$, $\alpha(\ell) \coloneqq (\alpha_1+\ell,0)$ and $\beta(\ell) \coloneqq (\beta_1+\ell,0)$.
According to Lemma~\ref{lem:estimatevarphi}, by using~\eqref{eq:estimateun} for $n=1$, we obtain, for all $(\ell,q) \in \N^2$,
\[
    \left\| \varphi_1 u_1 \right\|_{\alpha_1+\ell,\beta_1+\ell,q} \leqslant C_{\alpha_1,\beta_1,q} \frac{1}{r_t^{\frac{\beta_1-\alpha_1}{2} + \frac{7}{2}}} \left( \frac{\mathds{1}_{\lbrace\delta>0\rbrace}}{\delta^q} + \frac{\mathds{1}_{\lbrace\delta=0\rbrace}}{(t^\star-t)^q} \right) \left\| f \right\|_{\Sigma(\alpha_1 + \beta_1 + 2(\ell+q) + 5)}
\]
Therefore, we estimate
\begin{align*}
    \left\| \mathbf{p}_t^\dagger f_1 \right\|_{\alpha_1,\beta_1,p} \leqslant \frac{C_{\alpha_1,\beta_1,p} \| f_{\mathrm{in}} \|_{\Sigma(\alpha_1+\beta_1+2p+5)}}{r_t^{\frac{\beta_1 - \alpha_1}{2}}} \left( \frac{\mathds{1}_{\lbrace\delta>0\rbrace}}{\delta^p} + \frac{\mathds{1}_{\lbrace\delta=0\rbrace}}{(t^\star-t)^p} \right) & \Bigg( \int_0^t \frac{1}{r_s^3} \diff s + \left( \frac{1}{\delta^{-1}} + \frac{M_\star}{(t^\star-t)^{-1}} \right) \frac{1}{r_t^{3}} \Bigg).
\end{align*}
We conclude that~\eqref{eq:estimatefn} and~\eqref{eq:estimatefnM} hold for $n=1$.

\subsubsection{Induction step}

In this section, we finally prove Theorem~\ref{THM} by induction.
Let an integer $n \geqslant 2$ and let us assume that $\mathcal{P}(k)$ holds for all $k \in \llbracket 1,n-1 \rrbracket$.
Let us prove $\mathcal{P}(n)$.
Let $(\alpha,\beta) \in \N^2 \times \N^2$ and $p \in \N$.
Let us denote $K \coloneqq \alpha_2 + \beta_2$.
Let us first estimate $u_n$.
Since $A_0$ does not depend on $t$, according to Proposition~\ref{prop:estimateofT-1} and triangle inequality,
\[
    \| u_n \|_{\alpha,\beta,p} \leqslant C_{\alpha,\beta} \sum_{k=0}^{n-1} \left\| \partial_t^p A_{n-k}^t \Big( u_k + \mathbf{p}_t^\dagger f_k \otimes g_0 \Big) \right\|_{\alpha_1,\beta_1,\Lambda(K-1)} + \left\| \partial_t^p A_{n-k}^t \Big( u_k + \mathbf{p}_t^\dagger f_k \otimes g_0 \Big) \right\|_{\alpha_1,\beta_1+1,\Lambda(K-2)}.
\]
According to the induction assumption, in the previous sum, the worst estimate will always be generated by the second term since it has more derivatives in the first variable.
At the same time, for $k \in \llbracket 0,n-1 \rrbracket$, for a fixed semi-norm, the worst estimates between $u_k$ and $\mathbf{p}_t^\dagger f_k$ is given by $\mathbf{p}_t^\dagger f_k$.
Therefore, with an abuse of estimate, we write
\[
    \| u_n \|_{\alpha,\beta,p} \leqslant C_{\alpha,\beta} \sum_{k=0}^{n-1} \left\| \partial_t^p A_{n-k}^t \left( \mathbf{p}_t^\dagger f_k \otimes g_0 \right) \right\|_{\alpha_1,\beta_1+1,\Lambda(K-2)}.
\]
By using Proposition~\ref{prop:estimateAk} and Proposition~\ref{prop:estimateA1}, we obtain, after integration over the $x_2$ variable,
\begin{align*}
    \| u_n \|_{\alpha,\beta,p} \leqslant  C_{\alpha,\beta,p} \Bigg( & \sum_{q=0}^p \left| \partial_t^{p-q} \left( \frac{1}{r_t^{3/2}} \right) \right| \left( \left\| \mathbf{p}_t^\dagger f_{n-1} \right\|_{\alpha_1,\beta_1+1,q} + \left\| x_1 \partial_1 \mathbf{p}_t^\dagger f_{n-1} \right\|_{\alpha_1,\beta_1+1,q} \right) \\
    & + \frac{1}{r_t^{1/2}} \left\| \mathbf{p}_t^\dagger f_{n-1} \right\|_{\alpha_1,\beta_1+1,p+1} + \sum_{q=0}^p \left| \partial_t^{p-q} \left( \frac{1}{r_t^{1/2}} \right) \right| \left\| \partial_1 \mathbf{p}_t^\dagger f_{n-1} \right\|_{\alpha_1,\beta_1+1,q} \\
    & + \sum_{k=1}^{n-1} \sum_{q=0}^p \left| \partial_t^{p-q} \left( \frac{1}{r_t^{1 + \frac{n-k}{2}}} \right) \right| \underset{|\gamma| = n-k+1}{\sum_{\gamma \in \N^2}} \left\| x^\gamma \mathbf{p}_t^\dagger f_k \right\|_{\alpha_1,\beta_1+1,q} \Bigg).
\end{align*}
By using the induction assumption for all $k \in \llbracket 1,n-1 \rrbracket$, we conclude that~\eqref{eq:estimateun} and~\eqref{eq:estimateunM} hold for $n$.
Let us now estimate $\mathbf{p}_t^\dagger f_n$.
According to the decomposition of $f_n$ in~\eqref{eq:decomposef}, Lemma~\ref{lem:estimatephi} and triangle inequality,
\begin{align*}
    \left\| \mathbf{p}_t^\dagger f_n \right\|_{\alpha_1,\beta_1,p} \leqslant C_{\alpha_1,\beta_1,p} \sum_{k=1}^n & \Bigg( \frac{1}{r_t^{\frac{\beta_1 - \alpha_1}{2}}} \left( \frac{\mathds{1}_{\lbrace\delta>0\rbrace}}{\delta^p} + \frac{\mathds{1}_{\lbrace\delta=0\rbrace}}{(t^\star-t)^p} \right) \int_0^t r_s^{\frac{\beta_1 - \alpha_1 + 1}{2}} \sum_{q=0}^p \left\| \varphi_k u_{n+1-k}(s,\cdot) \right\|_{\alpha_1+q,\beta_1+q} \diff s \\
    & \quad + \sum_{q = 0}^{p-1} \left( \frac{\mathds{1}_{\lbrace\delta>0\rbrace}}{\delta^{p-q-3/2}} + \frac{\mathds{1}_{\lbrace\delta=0\rbrace}}{(t^\star-t)^{p-q-1-M_\star/2}} \right) \sum_{\ell = 0}^{p-1-q} \left\| \varphi_k u_{n+1-k} \right\|_{\alpha_1+\ell,\beta_1+\ell,q} \Bigg).
\end{align*}
Let us denote for all $\ell \in \N$, $\alpha(\ell) \coloneqq (\alpha_1+\ell,0)$ and $\beta(\ell) \coloneqq (\beta_1+\ell,0)$.
According to Lemma~\ref{lem:estimatevarphi}, for all $(\ell,q) \in \N^2$,
\begin{align*}
    \left\| \varphi_1 u_n \right\|_{\alpha_1+\ell,\beta_1+\ell,q} \leqslant C_{\alpha_1,\beta_1,q} & \Bigg( \sum_{j=0}^q \left| \partial_t^{q-j} \left( \frac{1}{r_t^{3/2}} \right) \right| \left( \left\| u_n \right\|_{\alpha(\ell),\beta(\ell),j} + \left\| x_1^2 u_n \right\|_{\alpha(\ell),\beta(\ell),j} \right) \\
    & \quad + \sum_{j = 0}^q \left| \partial_t^{q-j} \left( \frac{1}{r_t^{1/2}} \right) \right| \left( \left\| x_1 u_n \right\|_{\alpha(\ell),\beta(\ell),j} + \left\| \partial_1 u_n \right\|_{\alpha(\ell),\beta(\ell),j} \right) \Bigg),
\end{align*}
and for all $k \in \llbracket 2,n \rrbracket$,
\[
    \left\| \varphi_k u_{n+1-k} \right\|_{\alpha_1+\ell,\beta_1+\ell,q} \leqslant C_{\alpha_1,\beta_1,q,k} \sum_{j=0}^q  \left| \partial_t^{p-q} \left( \frac{1}{r_t^{1 + k/2}} \right) \right| \left( \left\| u_{n+1-k} \right\|_{\alpha(\ell),\beta(\ell),j} + \left\| x_1^{k+1} u_{n+1-k} \right\|_{\alpha(\ell),\beta(\ell),j} \right).
\]
By applying the induction assumption, estimates~\eqref{eq:estimateun} and~\eqref{eq:estimateunM} that we proved for $n$, we obtain, for $k \in \llbracket 1,n \rrbracket$,
\begin{itemize}
    \item if $\delta > 0$,
\[
    \left\| \varphi_k u_{n+1-k} \right\|_{\alpha_1+\ell,\beta_1+\ell,q} \leqslant C_{\alpha_1,\beta_1,q,k} \left\| f \right\|_{\Sigma(\alpha_1 + \beta_1 + 2(\ell + q) + 5n)} \frac{1}{\delta^{q + \frac{\beta_1 - \alpha_1}{2} + 3n + \frac{1}{2}}},
\]
which implies
\begin{align*}
    \left\| \mathbf{p}_t^\dagger f_n \right\|_{\alpha_1,\beta_1,p} \leqslant & C_{\alpha_1,\beta_1,p} \left\| f_{\mathrm{in}} \right\|_{\Sigma(\alpha_1 + \beta_1 + 2p + 5n)} \frac{1}{\delta^{p + \frac{\beta_1 - \alpha_1}{2}}} \Bigg( \int_0^t \frac{1}{F(s)^{3n}} \diff s + \frac{1}{\delta^{3n}} \Bigg).
\end{align*}
We conclude that~\eqref{eq:estimatefn} holds for $n$.
    \item if $\delta=0$,
\[
    \left\| \varphi_k u_{n+1-k} \right\|_{\alpha_1+\ell,\beta_1+\ell,q} \leqslant C_{\alpha_1,\beta_1,q,k} \left\| f \right\|_{\Sigma(\alpha_1 + \beta_1 + 2(\ell + q) + 5n)} \frac{1}{(t^\star - t)^{q + M_\star \frac{\beta_1 - \alpha_1}{2} + \frac{7}{2}M_\star + (3M_\star - 1)(n-1)}},
\]
which implies
\begin{align*}
    \left\| \mathbf{p}_t^\dagger f_n \right\|_{\alpha_1,\beta_1,p} \leqslant C_{\alpha_1,\beta_1,p} \left\| f_{\mathrm{in}} \right\|_{\Sigma(\alpha_1 + \beta_1 + 2p + 5n)} \frac{1}{(t^\star - t)^{p + M_\star\frac{\beta_1 - \alpha_1}{2}}} & \Bigg( \int_0^t \frac{1}{(t^\star - s)^{3M_\star + (3M_\star - 1)(n-1)}} \diff s \\
    & \quad + \frac{1}{(t^\star - t)^{(3M_\star - 1)n}} \Bigg).
\end{align*}
We conclude that~\eqref{eq:estimatefnM} holds for $n$.
\end{itemize}

\subsection{Proofs of Theorems~\ref{thm:distance}~and~\ref{theo:main}} \label{sec:PROOF}

We now prove Lemma~\ref{lem:an}.

\begin{proof}[Proof of Lemma~\ref{lem:an}]
Let $n \in \N\setminus\{0\}$.
According to~\eqref{eq:WP4},
\[
    (\eps \D_t + \H)\WP^\ep_{x_t,0} \left[ a^{(n)} \right] = \WP^\ep_{x_t,0} \left[ \sum_{k = 1}^n \sum_{j = n+1-k}^n \ep^{\frac{k+j+1}{2}} T_k^t a_j + \ep^{\frac{n+2}{2}} \sum_{j=0}^n \ep^{\frac{j}{2}} R_t^{n+1} a_j \right].
\]
Let $t \in [0,T)$.
Then using that for all $a \in \mathscr{C} \left( \R, \L \left( \R^2, \C^2 \right) \right)$, for all $t \in \R$,
\begin{equation}
	\| \WP_{x_t,0}^\varepsilon[a](t,.) \|_{\L(\R^2,\C^2)} = \| a(t,.) \|_{\L(\R^2,\C^2)}, \label{prop:WP1}
\end{equation}
and triangle inequality, we have
\[
\left\| (\eps \D_t + \H)\WP^\ep_{x_t,0} \left[ a^{(n)} \right] (t,\cdot) \right\|_{\L} \leqslant \sum_{k = 1}^n \sum_{j = n+1-k}^n \ep^{\frac{k+j+1}{2}} \left\| T_k^t a_j \right\|_{\L} + \ep^{\frac{n+2}{2}} \sum_{j=0}^n \ep^{\frac{j}{2}} \left\| R_t^{n+1} a_j \right\|_{\L}.
\]
According to Proposition~\ref{prop:TAA} and relation~\eqref{eq:defauf}, for all $(k,j) \in \llbracket 1, n \rrbracket \times \llbracket n+1-k, n \rrbracket$,
\begin{align*}
	\left\| T_k^t a_j \right\|_{\L} & = \sqrt{r_t} \left\| A_k^t \left( u_j + \mathbf{p}_t^\dagger f_j \otimes g_0 \right) \right\|_{\L}, \\
	\left\| R_t^{n+1} a_j \right\|_{\L} & = \sqrt{r_t} \left\| \widetilde{R}_t^{n+1} \left( u_j + \mathbf{p}_t^\dagger f_j \otimes g_0 \right) \right\|_{\L},
\end{align*}
where $ \widetilde{R}_t^{n+1} \coloneqq r_t^{-1/2} \mathscr{U}_t^{-1} R_t^{n+1} \mathscr{U}_t$.
Next, by using triangle inequality, according to Theorem~\ref{THM}, with an abuse of estimates, since the worst estimates are given by $\mathbf{p}_t^\dagger f_k$, we have
\begin{align*}
	\left\| (\eps \D_t + \H)\WP^\ep_{x_t,0} \left[ a^{(n)} \right] (t,\cdot) \right\|_{\L} \leqslant & \sqrt{r_t} \ep^{\frac{n+1}{2}} \sum_{k = 1}^n \ep^{\frac{k}{2}} \left\| A_k^t u_n \right\|_{\L} + \sqrt{r_t} \ep^{n+1} \left\| \widetilde{R}_t^{n+1} u_n \right\|_{\L} \\
	& \quad + \sqrt{r_t} \sum_{k = 2}^n \sum_{j = n+1-k}^{n-1} \ep^{\frac{k+j+1}{2}} \left\| A_k^t \left( \mathbf{p}_t^\dagger f_j \otimes g_0 \right) \right\|_{\L} \\
	& \quad + \sqrt{r_t} \ep^{\frac{n+2}{2}} \sum_{j=0}^{n-1} \ep^{\frac{j}{2}} \left\| \widetilde{R}_t^{n+1} \left( \mathbf{p}_t^\dagger f_j \otimes g_0 \right) \right\|_{\L},
\end{align*}
since we can consider $f_n = 0$.
Finally the proof follows by applying Proposition~\ref{prop:estimateAk}-\ref{prop:estimateA1}.
\end{proof}

Now we conclude by proving Theorem~\ref{thm:distance}.
The proof of Theorem~\ref{theo:main} follows from the same arguments and is therefore omitted.

\begin{proof}[Proof of Theorem~\ref{thm:distance}]
As $\Psi_0=\WP_{x_0,0}^\varepsilon[a^{(n)}]$, according to the energy estimate~\eqref{eq:energyestimate}, we have, for all $t\in \R$,
\begin{equation}\label{eq:esti1}
    \left\|\left(\Psi-\WP_{x_t}^\varepsilon[a^{(n)}]\right)(t,\cdot)\right\|_{\L}\le  \frac{1}{\ep} \int_0^t \left\| (\eps \D_s + \H) \WP^\ep_{x_s} \left[ a^{(n)} \right] (s,\cdot)\right\|_{\L} \diff s.
\end{equation}
Then we apply Lemma~\ref{lem:an} to expand~\eqref{eq:esti1}.
We finally apply Theorem~\ref{THM} to all the quantities of the right-hand side of Lemma~\ref{lem:an}.
By a straightforward computation, we have for all $\in s [0,t]$, 
\begin{equation}\label{eq:1}
    \ep^{\frac{n+1}{2}} \sqrt{\ep} \frac{1}{r_t} \left( \left\| u_n(s,\cdot) \right\|_{\L} + \left\| x_1 \partial_1 u_n(s,\cdot) \right\|_{\L} + \left\| x_2 \partial_2 u_n(s,\cdot) \right\|_{\L} \right)\leq C_n \varepsilon \left(\frac{\sqrt{\varepsilon}}{\delta^3}\right)^n\|f_{\rm in}\|_{\Sigma(5n)},
\end{equation}
and 
\begin{equation}\label{eq:2}
    \ep^{\frac{n+1}{2}} \sqrt{\ep} \Big( \left\| \partial_t u_n(s,\cdot) \right\|_{\L} + \left\| x_1 \partial_2 u_n(s,\cdot) \right\|_{\L} + \left\| x_2 \partial_1 u_n(s,\cdot) \right\|_{\L} \Big) \leq C_n \varepsilon \left(\frac{\sqrt{\varepsilon}}{\delta^3}\right)^n\|f_{\rm in}\|_{\Sigma(5n)}.
\end{equation}
Next by using a geometric sum argument, we have 
\begin{align}
    \sqrt{\frac{\ep}{r_t}} \ep^{\frac{n}{2}} \sum_{k = 1}^{n+1} \left( \frac{\ep}{r_t} \right)^{\frac{k}{2}} \underset{|\gamma| = k+1}{\sum_{\gamma \in \N^2}} \left\| x^\gamma u_n(s,\cdot) \right\|_{\L} &\le C_n \frac{\eps^{\frac{n}{2}}}{\delta^{3n-\frac{1}{2}}}\|f_{\rm in}\|_{\Sigma(6n)} \sum_{k = 1}^{n+1} \left( \frac{\ep}{\delta} \right)^{\frac{k}{2}}\nonumber\\
    &\leq C_n \varepsilon \|f_{\rm in}\|_{\Sigma(6n)} \frac{\varepsilon^{\frac{n}{2}+1}}{\delta^{3n}} \frac{1-\left(\frac{\varepsilon}{\delta}\right)^{\frac{n+1}{2}}}{1-\sqrt{\frac{\varepsilon}{\delta}}}\label{eq:3},
\end{align}
and as $\frac{1}{\delta^3}\ge \frac{1}{\sqrt{\delta}}$, we have 
\begin{align}
    \sqrt{\frac{\ep}{r_t}} \sum_{k = 2}^{n+1} \left( \frac{\ep}{r_t} \right)^{\frac{k}{2}} \sum_{j = n+1-k}^{n-1}  \ep^\frac{j}{2} \Big( &\left\| \mathbf{p}_t^\dagger f_j(s,\cdot) \right\|_{\L} + \left\| x_1^{k+1} \mathbf{p}_t^\dagger f_j(s,\cdot) \right\|_{\L} \Big) 
    \leq \|f_{\rm in}\|_{\Sigma(6n)} \sqrt{\frac{\eps}{\delta}}\sum_{k=2}^{n+1}\left(\frac{\eps}{\delta}\right)^{\frac{k}{2}}\sum_{j=n+1-k}^{n-1}\left( \frac{\sqrt{\varepsilon}}{\delta^{3}} \right)^j\nonumber\\
    &\le C_n \eps\|f_{\rm in}\|_{\Sigma(6n)} \left(\frac{\sqrt{\varepsilon}}{\delta^3}\right)^{n+1} \frac{1}
    {1-\frac{\sqrt{\varepsilon}}{\delta^3}}
    \sum_{k=2}^{n+1}
    \delta^{\frac{5k-7}{2}}
    \left(
    1-
    \left(\frac{\sqrt{\varepsilon}}{\delta^3}\right)^{k-1}
    \right)\nonumber\\
    &\le C_n \eps\|f_{\rm in}\|_{\Sigma(6n)} \left(\frac{\sqrt{\varepsilon}}{\delta^3}\right)^{n} \frac{1}
    {1-\frac{\sqrt{\varepsilon}}{\delta^3}}
    \sum_{k=2}^{n+1}
    \delta^{\frac{5k-9}{2}}
    \left(
    1-
    \left(\frac{\sqrt{\varepsilon}}{\sqrt{\delta}}\right)^{k-2}
    \right)\nonumber\\ 
    &\le C_n \eps\|f_{\rm in}\|_{\Sigma(6n)} \left(\frac{\sqrt{\varepsilon}}{\delta^3}\right)^{n} \left(\frac{1}
    {1-\frac{\sqrt{\varepsilon}}{\delta^3}}\right)^2
    \left(\frac{1-\left(\frac{\varepsilon}{\delta}\right)^{\frac{n}{2}}}{1-\left(\frac{\varepsilon}{\delta}\right)^{\frac{1}{2}}}\right)
    \label{eq:4}
\end{align}
Next, by Lemma~\ref{lem:an}, after summing~\eqref{eq:1} --~\eqref{eq:4} we obtain that 
\begin{align*}
    &\int_0^t \left\| (\eps \D_s + \H) \WP^\ep_{x_s} \left[ a^{(n)} \right] \right\|_{\L}\mathrm{d}s \leq  C_n~\|f_{\rm in}\|_{\Sigma(6n)}~t \left(\frac{\sqrt{\varepsilon}}{\delta^3}\right)^n\left(1+\left(\frac{\varepsilon}{\delta}\right)^\frac{n}{2}\right),
\end{align*}
which concludes the proof.
\end{proof}

\section*{Acknowledgments}
The authors are very grateful to Clotilde Fermanian Kammerer, Nicolas Raymond and Martin Averseng for many insightful discussions and valuable suggestions.
The authors acknowledge the support of the Région Pays de la Loire via the Connect Talent Project HiFrAn 2022 07750, and from the France 2030 program, Centre Henri Lebesgue ANR-11-LABX-0020-01 and the ANR project ANR-25-CE40-7296.

\appendix

\section{Spectrum of the symbol of the leading order operator} \label{app:spectral theory}

In this section, we prove Proposition~\ref{theo:spectrumt}.
In the following, we will omit the variable indices if this does not cause confusion.

\begin{proof}[Proof of Proposition~\ref{theo:spectrumt}] 
The operator $M(\xi)$ admits $\lambda(\xi)$ as an eigenvalue if and only if, there exists $f \in \L(\R,\C^2)$ normalized such that 
\begin{equation}\label{eq:syst BH hatA}
    \left\{ \begin{matrix} 2\xi f_1 & + & \mathfrak{a}f_2 & = & \lambda(\xi) f_1 \\ & & \mathfrak{a}^\dagger f_1  & = & \lambda(\xi) f_2 \end{matrix}\right .
\end{equation}
\begin{itemize}
\item Let us suppose that $\lambda(\xi) = 0$.
According to the injectivity of $\mathfrak{a}^\dagger$ on $\L(\R,\C)$ and the equality $\ker(\mathfrak{a}) = \Span( \mathfrak{h}_0)$, the function $(f_1,f_2)$ is normalized solution of~\eqref{eq:syst BH hatA} if and only if $(f_1,f_2) = (0,\mathfrak{h}_0)$.
\item Let us suppose that $\lambda(\xi) \neq 0$.
Equation~\eqref{eq:syst BH hatA} is equivalent to 
\[
    \left\{ \begin{matrix} \mathfrak{a}\mathfrak{a}^\dagger f_1 & = & \left( \lambda(\xi)^2 - 2 \xi \lambda(\xi) \right) f_1 \\ f_2 & = & \frac{1}{\lambda(\xi)}\mathfrak{a}^\dagger f_1 \end{matrix}\right.
\]
According to Proposition~\ref{prop:propriété création anihilation}, the function $(f_1,f_2)$ is a normalized solution of~\eqref{eq:syst BH hatA} if, and only if, there exists $n \in \N\setminus\{0\}$ such that $\lambda(\xi)$ solves
\[
    \lambda(\xi)^2 - 2 \xi \lambda(\xi) = 2n,
\]
and the function $(f_1,f_2)$ satisfy
\[
     \begin{pmatrix} f_1(\xi) \\ f_2(\xi) \end{pmatrix}  = \alpha_n (\xi) \begin{pmatrix} \mathfrak{h}_{|n|-1} \\ \displaystyle\frac{\sqrt{2|n|}}{\xi+\sgn(n) \sqrt{\xi^2 + 2|n|}} \mathfrak{h}_{|n|}\end{pmatrix},
\]
with $\displaystyle{\alpha_n (\xi) \coloneqq \frac{1}{\sqrt{2}} \sqrt{ 1 + \frac{\xi}{\lambda_n(\xi) - \xi} }}$ chosen for the normalization.
Finally, the discrete spectrum of $M(\xi)$ is given by
\[
    \operatorname{Sp}_{\mathrm{disc}}\left(M(\xi)\right) = \{ 0 \} \cup \left\{\lambda_n(\xi) \coloneqq \xi + \sgn(n) \sqrt{\xi^2 + 2|n|} \ \mid n \in \Z_{\neq 0} \right\}.
\]
\end{itemize}
It remains to prove that the family $\displaystyle \left(g_n(\xi)\right)_{n\in\Z}$ is an Hilbertian basis of $\L(\R,\C^2)$.
\begin{itemize}
    \item[$\bullet$] The normality of the family is a consequence of the computation of $\alpha_n(\xi)$.
    \item[$\bullet$] The orthogonality and the completeness are inherited from the orthogonality and the completeness of the family $(\mathfrak{h}_{n})_{n\in\N}$ in $\L(\R,\C)$.
\end{itemize}
We conclude thanks to the spectral theorem since $(M(\xi),\mathcal{B}^1(\R,\C^2))$ is self-adjoint with compact resolvent.
\end{proof}

\section{Equivalence of \texorpdfstring{$\L$}{}-Schwartz semi-norms}

In this appendix, we prove some $\L$ equivalence norms.
Theses results may also be deduced from~\cite[Section 2]{BenAbdallahCatellaMehats} and~\cite{Helffer}.

\begin{lemma}[Equivalence of Schwartz semi-norms $\Sigma_K$ and $\Lambda_K$]\label{lem:equivalencenorm} For all integer $K$, $\| \cdot \|_{\Sigma(K)}$ and $\| \cdot \|_{\Lambda(K)}$ are equivalent.
\end{lemma}

\begin{proof}
Let $f \in \mathscr{S}(\R,\C)$.
Let $n \in \N$ and let us define $\mathcal{K}(n) \subset \N^3$ and $\mathcal{L}(n) \subset \N^3$ by
\begin{align*}
    \mathcal{K}(n) & \coloneqq \left\{ (i,j,k) \in \N^3 \mid i+j+k \leqslant n \ \text{and} \ (i = 0 \ \text{or} \ j = 0) \right\}, \\
    \mathcal{L}(n) & \coloneqq \left\{ (a,b) \in \N^2 \mid a+b \leqslant n \right\}.
\end{align*}
According to identity $\mathfrak{a}\mathfrak{a}^\dagger = (1 + x^2 - \partial_x^2)$, we have the following, as quadratic form,
\begin{align*}
    1 & \leqslant \mathfrak{a}\mathfrak{a}^\dagger, \quad \quad \text{then} \ \left( \left( \mathfrak{a}\mathfrak{a}^\dagger \right)^{n/2} \right)_{n \in \N} \ \text{is non-decreasing}, \\
    x^2 & \leqslant \mathfrak{a}\mathfrak{a}^\dagger, \quad \quad \text{then, for all $n \in \N$,} \ \left\| x^nf \right\|_{\L(\R,\C)} \leqslant  \left\|(\mathfrak{a}\mathfrak{a}^\dagger)^{n/2}f\right\|_{\L(\R,\C)}, \\
    -\partial_x^2 & \leqslant \mathfrak{a}\mathfrak{a}^\dagger, \quad \quad \text{then, for all $n \in \N$,} \ \left\| f^{(n)} \right\|_{\L(\R,\C)} \leqslant  \left\|(\mathfrak{a}\mathfrak{a}^\dagger)^{n/2}f\right\|_{\L(\R,\C)}.
\end{align*}
\begin{itemize}
    \item Let us prove there exists $C_n > 0$ such that for all $f \in \mathscr{S}(\R,\C)$, $\|f\|_{\Sigma(n)} \leqslant C_n \|f\|_{\Lambda(n)}$.
        Since,
        \[
            x = \frac{\mathfrak{a} + \mathfrak{a}^\dagger}{2}, \quad \quad \partial_x = \frac{\mathfrak{a} - \mathfrak{a}^\dagger}{2}, \quad \quad \left[ \mathfrak{a},\mathfrak{a}^\dagger \right] = 2, 
        \]
        then, by re-ordering the terms in the binomial expansion, for all $(a,b) \in \N^2$ with $a+b \leqslant n$, there exists $\left( c^{(a,b)}_{i,j,k} \right)_{(i,j,k) \in \mathcal{K}(n)} \in \R_+^{\mathcal{K}(n)}$ such that
        \[
            x^a\partial_x^b = \sum_{(i,j,k) \in K} c^{(a,b)}_{i,j,k} \mathfrak{a}^i \left( \mathfrak{a}^\dagger \right)^j \left( \mathfrak{a} \mathfrak{a}^\dagger \right)^{k/2}.
        \]
        But, by spectral property,
        \[
            \left\| \mathfrak{a}^\dagger f \right\|_{\L(\R,\C)} \leqslant \left\| (\mathfrak{a}\mathfrak{a}^\dagger)^{1/2} f \right\|_{\L(\R,\C)}, \quad \quad \left\| \mathfrak{a} f \right\|_{\L(\R,\C)} \leqslant \left\| (\mathfrak{a}^\dagger\mathfrak{a})^{1/2} f \right\|_{\L(\R,\C)},
        \]
        then, for all $(i,j,k) \in \mathcal{K}(n)$,
        \[
            \left\| \mathfrak{a}^i \left( \mathfrak{a}^\dagger \right)^j \left( \mathfrak{a} \mathfrak{a}^\dagger \right)^{k/2} f \right\|_{\L(\R,\C)} \leqslant C_{i,j,k} \left\|  f \right\|_{\Lambda(i+j+k)} \leqslant C_n \left\|  f \right\|_{\Lambda(n)},
        \]
        so
        \[
            \left\| x^a\partial_x^b f \right\|_{\L(\R,\C)} \leqslant C_n \|f\|_{\Lambda(n)}.
        \]
    \item Let us prove there exists $C_n > 0$ such that for all $f \in \mathscr{S}(\R,\C)$, $\|f\|_{\Lambda(n)} \leqslant C_n \|f\|_{\Sigma(n)}$.
        Since
        \[
            \mathfrak{a} = x + \partial_x, \quad \quad \mathfrak{a}^\dagger = x - \partial_x, \quad \quad \left[ x, \partial_x \right] = -1, 
        \]
        then, according to the binomial identity and to Leibniz formula, there exists $(c_{a,b})_{(a,b) \in \mathcal{L}(n)}$ such that
        \[
            \left( \mathfrak{a}\mathfrak{a}^\dagger \right)^n = \sum_{(a,b) \in L(n)} c_{a,b} x^a \partial_x^b,
        \]
        then, as quadratic form,
        \[
            \left( \mathfrak{a}\mathfrak{a}^\dagger \right)^{n/2} \leqslant C_n \sum_{(a,b) \in \mathcal{L}(n)} \langle x \rangle^a \langle i\partial_x \rangle^b.
        \]
        We conclude by equivalence of the families of norm $\left( \left\| \langle x \rangle^a \langle i\partial_x \rangle^b \cdot \right\| \right)_{(a,b) \in \mathcal{L}(n)}$ and the norm $\|.\|_{\Sigma(n)}$.
\end{itemize}
\end{proof}

\bibliographystyle{plain}
\bibliography{biblio}

\bigskip

\noindent (N. Frantz) \textbf{Univ Angers}, CNRS, \textbf{LAREMA}, SFR MATHSTIC, F-49000 Angers, France,\\
\textit{Email adress :} \texttt{nicolas.frantz@univ-angers.fr} \\
(\'E. Vacelet) \textbf{Univ Angers}, CNRS, \textbf{LAREMA}, SFR MATHSTIC, F-49000 Angers, France,\\
\textit{Email address :} \texttt{eric.vacelet@univ-angers.fr}

\end{document}